\documentclass[11pt,a4paper]{amsart}
\usepackage{amsaddr}
\usepackage{amsfonts}
\usepackage{amsthm}
\usepackage{amsmath}
\usepackage{amssymb}
\usepackage{amscd}
\usepackage[latin2]{inputenc}
\usepackage{t1enc}
\usepackage[mathscr]{eucal}
\usepackage{indentfirst}
\usepackage{graphicx}
\usepackage{graphics}
\usepackage{pict2e}
\usepackage{epic}
\usepackage{bm}
\numberwithin{equation}{section}
\usepackage[margin=2.9cm]{geometry}
\usepackage{epstopdf} 
\usepackage{upref}
\usepackage{hyperref}
\usepackage{bbm}
\usepackage{setspace}
 \usepackage[T1]{fontenc}
 \usepackage{lmodern}

\theoremstyle{plain}
\newtheorem{Th}{Theorem}[section]
\newtheorem{Lem}[Th]{Lemma}

\newtheorem{Prop}[Th]{Proposition}

 \theoremstyle{definition}
\newtheorem{Def}[Th]{Definition}

\newtheorem{Rem}[Th]{Remark}

\newtheorem{thma}{Theorem}

\newtheorem{lema}{Lemma}

\newtheorem{propa}{Proposition}

\newcommand{\be} {\begin{equation}}
\newcommand{\ee} {\end{equation}}

\newcommand{\rt}{\mathbb{R}^{2}}

\newcommand{\f}{\mathcal{F}}
\newcommand{\g}{\Gamma^{\bm{\beta}}}

\newcommand{\si}{\sum_{i=1}^n}
\newcommand{\sj}{\sum_{j=1}^n}
\newcommand{\inte}{\int_{\rt}}
\newcommand{\rk}{\bm{\rho}_{\tau}^k}
\newcommand{\rka}{\bm{\rho}_{\tau}^{k-1}}
\newcommand{\rki}{\rho_{\tau,i}^k}
\newcommand{\rkai}{\rho_{\tau,i}^{k-1}}
\newcommand{\rkim}{\rho_{\tau_m,i}^k}
\newcommand{\rkaim}{\rho_{\tau_m,i}^{k-1}}

\newcommand{\rl}{\bm{\rho}_{\tau}^l}
\newcommand{\rla}{\bm{\rho}_{\tau}^{l-1}}
\newcommand{\rli}{\rho_{\tau,i}^l}
\newcommand{\rlai}{\rho_{\tau,i}^{l-1}}

\newcommand{\rti}{\rho_{\tau,i}}

\newcommand{\rkj}{\rho_{\tau,j}^k}

\newcommand{\fs}{\si\inte \rho_i(x)\ln \rho_i(x) dx + \si\sj\frac{a_{ij}}{4\pi}\inte \inte \rho_i(x)
\ln|x-y|\rho_j(y)dxdy}

\begin{document}

\title[Patlak-Keller-Segel system]{On Patlak-Keller-Segel system for several populations: 
A Gradient flow approach}
\author{Debabrata Karmakar and Gershon Wolansky} 

\address{Technion, Israel Institute of Technology, 32000 Haifa, Israel}

\email{dkarmaker@campus.technion.ac.il, gershonw@math.technion.ac.il}

\subjclass[2010]{Primary: 35K65, 35K40, Secondary: 35Q92}

\keywords{Chemotaxis for multi-species; Patlak-Keller-Segel system; Minimizing movement scheme;
Wasserstein distance}

\begin{abstract} 
We study the global in time existence of solutions to the parabolic-elliptic 
Patlak-Keller-Segel system of multi-species populations. We prove that if the initial mass 
satisfies an appropriate notion of sub-criticality, then the system has a solution defined for 
all time. We explore the gradient flow structure in the {\it Wasserstein space} to study the 
question of existence. Moreover, we show that the obtained solution satisfies energy dissipation 
inequality.
\end{abstract}

\maketitle
{\small\tableofcontents}

\section{Introduction} 
In this article, we study the global existence and uniqueness  of solutions to
the parabolic-elliptic Patlak-Keller-Segel system (PKS-system in short) 
for $n$ populations
interacting via a self produced 
chemical agent on the two-dimensional Euclidean space $\rt.$ The evolution of these living cells are
governed by the following system of equations
\begin{equation}\label{kss}
\begin{cases}
& \partial_t \rho_i(x,t) = \Delta_x \rho_i(x,t) - \sj a_{ij}\nabla_x \cdot 
\left(\rho_i(x,t)\nabla_x u_j(x,t)\right), \ in \ \rt \times (0,\infty),\\
&-\Delta_x u_i(x,t) = \rho_i(x,t), \hspace{5.8 cm} \ \ in \  \rt \times (0,\infty), \\
&\  \rho_i(x,0) = \rho_i^0 , \ \ i=1,\ldots, n,
\end{cases}
\end{equation}
where $\rho_i(x,t)$ denotes the cell density of the $i$-th population, $ u_i(x,t)$ denotes the 
concentration of the chemical, called {\it chemoattractant}, produced by the $i$-th population and $a_{ij}$ are  
constants denoting the sensitivity of the $i$-th population towards the chemical gradient produced
by the $j$-th population and $\rho_i^0$ is the initial cell distribution of the $i$-th population. 
Since the solutions to the  
Poisson equation $-\Delta u = \rho$ is unique up to a harmonic function, we define
concentration of the chemoattractant $u_i$ by the Newtonian potential of $\rho_i$
\begin{align}\label{newtonian potential}
 u_i(x,t) = -\frac{1}{2\pi}\inte \ln|x-y|\rho_i(y,t) \ dy, \ \ \ i=1,\ldots,n.
\end{align}

The sensitivity parameters 
$a_{ij}>0$ meaning that the $i$-th population is attracted to  the chemoattractant 
produced by the $j$-th population and tends to climb up its gradient.  
On the other hand, $a_{ij}<0$ meaning that $i$-th population is repelled from  the same and 
tends to climb down its gradient. In particular, $ a_{ii} > 0$  (or  $a_{ii} < 0$) is 
the condition of self-attraction (or, self-repulsion) 
of the population $i.$ The case $a_{ij}a_{ji} <0$ is the unhappy situation between the 
$i$-th and $j$-th population and is the origin of {\it conflict of interests}. 
In this article, we assume the sensitivity matrix $(a_{ij})$ is symmetric with non-negative entries 
$a_{ij}\geq 0, \ for \ all \ i,j$, that is the {\it conflict free} case.

The manifestation of a single population or the scalar case $n=1$ (where $a:= a_{11}$) has been the subject of 
intensive research over the past couple of decades. 
See \cite{Pa, KS, W1, Horsurvey1, H} for the biological motivations.  

Equation \eqref{kss} is a typical example of a conservative drift-diffusion equation. 
The smoothing effect induced by the diffusion term $\Delta \rho_i$ and the weighted cumulative
drift induced by the chemical gradients $\sj a_{ij}\nabla u_j$ which assists the cells to accumulate, 
compete against each other. It is well understood, at least for the scalar case, 
the $L^1$-norm of the initial datum is a salient parameter which separates the dichotomy between 
the global in time existence and the chemotactic collapse (or, finite time blow up). 
More precisely,
if the initial number of bacteria is smaller than the critical threshold $\inte \rho_0 < 8\pi/ a$ then 
the process of aggression counterbalanced by the diffusion \cite{BD}. However, if it crosses the critical 
threshold, i.e., $\inte \rho^0 > 8\pi/ a$ the production of the chemical agent attracting the cells 
increase so as that the diffusion can no longer compete against the drift force, resulting in an inevitable 
chemotactic collapse. The critical case $ \inte \rho^0 = 8\pi/a$ is the perfect balance between these 
two opposing forces. 
In this case, a solution exists globally in time \cite{BCM, BKLN}, but if the second moment 
of the initial data is finite, then the solutions concentrate
in the form of a Dirac delta measure as time $t$ goes to infinity \cite{BCM}. We refer the 
interested readers to \cite{Childress, NSenba, SeSu, SSbook, Subook, BT, BKLN, BD, BCM, 
BDP, BCCjfa, CD, FM} and the references therein devoted to the study of parabolic-elliptic PKS-system 
for single population,
and also excellent survey articles \cite{Horsurvey1, Horsurvey2, Blanchet} in this regard.

A solution $\bm{\rho} :=(\rho_1, \ldots, \rho_n)$ to the PKS-system \eqref{kss},
at least formally, possess the following fundamental identities: 
\begin{itemize}
 \item Conservation of mass:
 \begin{align} \label{mass Conservation}
   \inte \bm{\rho}(x,t) \ dx = \inte \bm{\rho}^0(x) \ dx = \bm{\beta}, \ \ for \ all \ t >0.
 \end{align}
\item Free energy dissipation or the free energy identity:
\begin{align} \label{free energy equality}
 \f(\bm{\rho}(\cdot , t)) + \int_0^t \mathcal{D}_{\f}(\bm{\rho}(\cdot, s)) \ ds = \f(\bm{\rho}^0),
\end{align}
\end{itemize}
where the free energy $\f$ is defined by
 \begin{align} \label{free energy}
 \f(\bm{\rho}) = \fs,
\end{align}
and the dissipation of free energy $\mathcal{D}_{\f}$ is defined by
\begin{align} \label{dissipation of f}
 \mathcal{D}_{\f}(\bm{\rho}) = \si  \inte \left|\frac{\nabla \rho_{i}(x)}{\rho_{i}(x)} 
 - \sj a_{ij}\nabla u_{j}(x)\right|^2 \rho_{i}(x) \ dx.
\end{align}
Moreover, if the second moment of the initial condition $M_2(\bm{\rho}^0) :=\si \inte|x|^2\rho^0_i(x) \ dx$ 
is finite then formally 
\begin{align}\label{second moment} 
 M_2(\bm{\rho}(\cdot,t)) \ dx = \frac{\Lambda_{I}(\bm{\beta})}{2\pi}t + M_2(\bm{\rho}^0),
\end{align}
where $I = \{1,\ldots, n\}$ and $\Lambda_I(\bm{\beta})$ is a quadratic polynomial in $\bm{\beta}$ defined by
\begin{align}\label{lambda}
 \Lambda_J(\bm{\beta}) := \sum_{i \in J} \beta_i \left(8\pi - \sum_{j \in J} a_{ij}\beta_j\right), \ \ for \ all \ \emptyset \neq J 
 \subset I.
\end{align}

In particular, when $n=1, \ \Lambda_{\{1\}}(\beta) = \beta(8\pi - a_{11} \beta).$
The critical constant $8\pi$ in single population emerges from the
time evolution of the
second moment \eqref{second moment}. However, the proof of existence of global in time solution 
in the sub-critical case is much more delicate issue and has been explored in \cite{BD} using the
energy method. One of the fundamental tool in their analysis is the
logarithmic Hardy-Littlewood-Sobolev (HLS) inequality in $\rt:$
\begin{align*}
 \f_{scalar}(\rho) := \inte \rho(x)\ln \rho(x) \ dx + \frac{a}{4\pi}\inte\inte \rho(x)\ln|x-y|\rho(y) \ dxdy
 \geq -C
\end{align*}
for all $\rho \in \Upsilon^{\beta}:=\{ \rho \in L^1_+(\rt)| \ 
\inte \rho \ln \rho <+\infty, \inte \ln(1+|x|^2)\rho < +\infty, \inte \rho = \beta\}$ if 
and only if $\beta = 8\pi/a.$ Roughly speaking, if $\inte \rho_0 < 8\pi/a$ and the 
total mass ($\inte \rho(x,t) \ dx$) being conserved for all time $t,$ the logarithmic HLS-inequality
gives an bound on the entropy of the solution ($\inte \rho \ln \rho$), which prevents 
the solutions to blow-up in finite time. 

Of particular, our interest lies to the alternative observation by Otto \cite{Otto} who saw that these 
class of PDEs \eqref{kss} inherits gradient flow structure in the space of probability
measures with respect to an appropriate metric. 
In their subsequent works, Jordan-Kinderlehrer and Otto implemented this idea
in the context of the heat and the Fokker-Planck equation \cite{JKO}.

We observe formally that the system \eqref{kss} can be written as
\begin{align*}
 \partial_t \rho_i = \nabla \cdot \left(\rho_i \nabla \frac{\delta \f}{\delta \rho_i}(\bm{\rho})\right), \ \ i=1,\ldots,n.
\end{align*}
This is the formal structure of a gradient flow of the free energy $\f$
in the space $\mathcal{P}^{\beta_1}(\rt)\times \cdots \times \mathcal{P}^{\beta_n}(\rt)$ 
equipped with the 2-Wasserstein distance $\bm{d}_{\mbox{{\tiny{W}}}}$ (see section $2$ for definition),
where $\mathcal{P}^{\beta_i}(\rt)$ denotes the space of non-negative Borel measures on $\rt$ with total mass 
$\beta_i$ and
$\frac{\delta \f}{\delta \rho_i}$ denotes the first variation of the functional $\f$ with respect to the 
variable $\rho_i.$
The functional $\f$ on the product space 
$\mathcal{P}^{\beta_1}(\rt)\times \cdots \times \mathcal{P}^{\beta_n}(\rt)$ is defined by $\f(\bm{\rho})$
if $\bm{\rho} \in \g,$  where 
\begin{align*}
 \g = \left\{\bm{\rho}= (\rho_i)_{i=1}^n | \ \rho_i \in L^1_+(\rt),
\inte \rho_i(x)\ln \rho_i(x) \ dx < +\infty, 
\inte \rho_i(x)\ dx = \beta_i, \right. \\
 \left.  \ \inte \rho_i(x)\ln(1+|x|^2) \ dx<+\infty\right\}
\end{align*}
and $+\infty$ elsewhere. 

The study of gradient flows in general metric spaces is substantially a vast subject and have been pioneered by 
Ambrosio, Gigli and Savar\'{e} in their book \cite{AGS}. However, to make sense of gradient flows
in a general metric space and establishing a complete existence-uniqueness theory requires a certain convexity 
assumption on the functional (for example $\lambda$-geodesic convexity) as well as on the metric 
(for example $C^2G^2$-condition).
The functional $\f$ does not possess such convexity property (it is neither convex in the usual sense nor,
displacement convex in the sense of McCann \cite{McCann}) and hence different approach is necessary.
Instead, we will rely on the time-discretized variational formulation introduce in \cite{JKO},
known as minimizing movement scheme or, JKO-scheme and the functional analytic framework 
to study the convergence of the scheme:
for a time step $\tau>0,$ define
\begin{align}\label{formal jko}
 \bm{\rho}^{k}_{\tau} \in \arg\min_{\bm{\rho}\in \g_2} \left(\f(\bm{\rho}) +\frac{1}{2\tau} \bm{d}_{\mbox{{\tiny{W}}}}^2
 (\bm{\rho}, \bm{\rho}^{k-1}_{\tau})\right), \ \ \ k \geq 1
\end{align}
with $\bm{\rho}^0_{\tau} = \bm{\rho}^0,$ where 
\begin{align*}
 \g_2:=\left\{ \bm{\rho} \in \g \ | \ M_2(\bm{\rho}) := \si \inte |x|^2\rho_i \ dx < +\infty \right\},
\end{align*}
{\it and throughout this article we will assume the initial condition}
$\bm{\rho}^0 \in \g_2.$ 

The existence of a minimizer to  (\ref{formal jko}) depends on several relations depending on $\bm{\beta}$ 
and the interaction matrix $(a_{ij}).$ More precisely,
the boundedness from below of $\f$ and the functional $\bm{\rho}
\longmapsto \f(\bm{\rho}) + \frac{1}{2\tau}\bm{d}_{\mbox{{\tiny{W}}}}^2(\bm{\rho}, \bm{\eta})$ depends on the 
the following optimal relations: 
\begin{align} \label{beta}
 \begin{cases}
  \Lambda_J(\bm{\beta}) \geq 0, \ \ for \ all \ \emptyset \neq J 
 \subset \{1,\ldots, n\}, \\
 \mbox{ {\it if  for  some}} \ J, \ \Lambda_J(\bm{\beta}) = 0, \ then \ a_{ii} + \Lambda_{J\backslash \{i\}}(\bm{\beta})>0, \ \forall 
 i \in J
 \end{cases}
\end{align}
where $\Lambda_J({\bm{\beta}})$ is defined by \eqref{lambda}.
In particular, it is shown in \cite{CSW,SW} that 
$\Lambda_I(\bm{\beta}) = 0$ and \eqref{beta} is necessary and sufficient condition 
for the boundedness from below of $\f$ over $\g.$ On the other hand, if $\bm{\eta} \in \g_2,$ then 
\eqref{beta} is necessary and sufficient for  $\f(\cdot) + \frac{1}{2\tau}\bm{d}_{\mbox{{\tiny{W}}}}^2
(\cdot, \bm{\eta})$ to be bounded from below on $\g_2.$  

The existence of minimizers in \eqref{formal jko} is, however, a delicate question. To give a
glimpse of this question let us mention some results from \cite{KW}.
In \cite{KW} we considered the case $\bm{\eta} = (\beta_1\delta_{v_1}, \ldots, \beta_n\delta_{v_n})$
where $v_i$ are arbitrary vectors in $\rt$ and $\delta_{v_i}$ is the Dirac measure at $v_i.$ We
found that if $\bm{\beta}$ is {\it sub-critical}, that is $\Lambda_J(\bm{\bm{\beta}}) > 0$
for all $J$ then there always exists a minimizer. However, if $\bm{\beta}$ is {\it critical}, that is 
$\Lambda_I(\bm{\beta}) = 0$ and $\Lambda_J(\bm{\beta})>0$ for all $J \neq I,$ then the existence 
of a minimizer depends on the gap between $v_i$s.
In particular, for $n=2$ and $\bm{\beta}$ critical we showed that if $v_1 = v_2$ then no 
minimizer exists, on the other hand if $|v_1-v_2|$ is large enough then a minimizer does exist.
The optimal gap for which a minimizer exists is still an open question.
However, we do not address such questions in this article
and only consider the case $\bm{\beta}$ sub-critical.
\begin{Def}[Sub-critical] 
$\bm{\beta}$ is said to be sub-critical if and only if
\begin{align*}
 \Lambda_J(\bm{\beta})  > 0, \ \ for \ all \ \emptyset \neq J \subset I.
\end{align*}
 \end{Def}
 Before proceeding further let us first introduce the appropriate notion of a weak solution
 to the PKS-system \eqref{kss}. Throughout this article, we use the notation $\mathcal{H}(\bm{\rho})
 := \si \inte \rho_i \ln \rho_i$ to denote the entropy of the solutions and 
 $\mathcal{H}_+(\bm{\rho})
 := \si \inte \rho_i (\ln \rho_i)_+$ is the positive part of the entropy.
 \begin{Def} \label{weak sol}
  For any initial data $\bm{\rho}^0$ in $\g_2$ and $T^*>0$ we say that 
  a nonnegative vector valued function $\bm{\rho} \in \left(C([0,T^*); \mathcal{D}^{\prime}(\rt))\right)^n$
  satisfying
  \begin{align*}
   \mathcal{A}_T(\bm{\rho}) :=\sup_{t \in [0,T]} \Big(\mathcal{H}_+(\bm{\rho}(t)) + M_2(\bm{\rho}(t))\Big)
+ \int_0^T \mathcal{D}_{\f}(\bm{\rho}(t)) \ dt < +\infty, \ \forall \ T \in (0,T^*)
  \end{align*}
is a weak solution to the PKS-system \eqref{kss} on the time interval $(0,T^*)$
associated to the initial condition $\bm{\rho}^0$ if $\bm{\rho}$ satisfies \eqref{mass Conservation}
and 
 \begin{align*}
  &\int_0^{T^*} \inte \partial_t\xi(x,t) \rho_i(x,t) \ dxdt + \inte \xi(x,0) \rho_i^0(x) \ dx \\
 &-\int_{0}^{T^*} \inte  \rho_i(x,t)\left( \frac{\nabla_x \rho_i(x,t)}{\rho_i(x,t)}  
 -\sj a_{ij}  \nabla_x u_j(x,t)\right) \cdot \nabla_x \xi(x,t) \ dxdt = 0
 \end{align*}
  for all $\xi \in C_c^{2}([0,T^*)\times \rt)$ and for all $i=1,\ldots,n.$ If $T^* = +\infty$ we say $\bm{\rho}$ is a global
  weak solution of the system.
  \end{Def}

Notably, by virtue of finite energy dissipation and Cauchy-Schwartz inequality all the 
terms in the weak formulation of \eqref{kss} makes sense.
The gradient flow structure for a single population PKS-system $n=1$ with sub-critical 
mass has been explored earlier in \cite{JKO2}.
However, there is an essential difference between the gradient flow structure of the single population 
and that of multi-population, as explained below:

The first difficulty arises in pursuance of  the Euler-Lagrange equation 
for the variational scheme \eqref{formal jko}. For the minimization problem of type \eqref{formal jko}, 
is standard (in the theory 
of optimal transport) to consider the variation $\bm{\rho}_{\epsilon} = \bm{T}_{\epsilon}\#\rk$
of the minimizer $\rk$ and compute $\lim_{\epsilon \rightarrow 0+}\frac{1}{\epsilon}
\left(\f(\bm{\rho}_{\epsilon}) - \f(\rk) + \frac{1}{2\tau}(\bm{d}_{\mbox{{\tiny{W}}}}^2(\bm{\rho}_{\epsilon}, \rka)
- \bm{d}_{\mbox{{\tiny{W}}}}^2(\rk,\rka))\right) \geq 0,$
where $\bm{T}_{\epsilon}$ is a $C^{\infty}$ diffeomorphism of $(\rt)^n$
and $\bm{T}_{\epsilon}\#$ denotes the push forward of a measure under $\bm{T}_{\epsilon}.$
In particular, we are interest in  
$\bm{T}_{\epsilon}$ of the form $\bm{T}_{\epsilon} = \bm{I} + \epsilon \bm{\zeta}, \ \bm{\zeta} = (\zeta_1,\ldots, \zeta_n)$
and $\zeta_i \in C_c^{\infty}(\rt;\rt).$ This is where the first major difficulty arises. The 
minimizing movement scheme contains terms of three types: the entropy term, interaction energy
term and the Wasserstein distance term. For $\bm{T}_{\epsilon}$ is being chosen of this 
particular form, we can evaluate the limits
\begin{align*}
 \frac{\mathcal{H}(\bm{\rho}_{\epsilon}) - \mathcal{H}(\rk)}{\epsilon}, \ \ \ \ 
 \frac{\bm{d}_{\mbox{{\tiny{W}}}}^2(\bm{\rho}_{\epsilon}, \rka)
- \bm{d}_{\mbox{{\tiny{W}}}}^2(\rk,\rka)}{2\epsilon}
\end{align*}
where $\epsilon\rightarrow 0+$. However, the interaction energy contains terms of the form
\begin{align}\label{dificulty}
 \frac{1}{\epsilon}\inte\inte \left(\ln|x-y+\epsilon(\zeta_i(x)-\zeta_j(y))| - \ln|x-y|\right)\rki(x)\rkj(y) dxdy.
\end{align}
To see the difficulty, for $i \neq j$ we choose $\zeta_j \equiv 0$ to obtain formally in the limit 
(up to a constant factor times) $\inte \zeta_i(x) \cdot \nabla u_{\tau,j}^k(x) \rho_{\tau,i}^k(x),$ as 
$\epsilon \rightarrow 0+$.
The a-priori bound on the entropy of  $\rki$ implies that   the Newtonian potentials 
$u_{\tau,j}^k$ belongs to $H^1_{loc}(\rt)$ only, and so this limit is not properly defined.  
On the other hand if $i=j$ (as in the case $n=1$) then it is easy to pass to the limit 
$\epsilon \rightarrow 0+$ 
\begin{align*}
 \inte \inte \frac{(\zeta_i(x) - \zeta_i(y))\cdot (x-y)}{|x-y|^2} \rki(x)\rki(y) \ dxdy,
\end{align*}
which is well defined because $\zeta_i$ is smooth and compactly supported. To make sense of the terms 
corresponding to $ i \neq j,$  we  need higher regularity of the solution $\rki.$ In precise, $\rki$ 
must be at least $L^2(\rt)$ regular. We adopt the flow interchange technique introduced by 
Matthes, McCann and Savar\'{e} \cite{MMS} in spirit of \cite{BL13, Hybrid}
to tackle this issue. 
The next difficulty arises in passing to a limit in the Euler-Lagrange equation mentioned above. 
For the convergence of the scheme, we need a uniform in time estimates on the Newtonian potentials. 
It follows from the scheme that  the time interpolates have $\frac{1}{2}$-H\"{o}lder estimates in time with 
respect to the Wasserstein distance. The H\"{o}lder regularity only gives the continuity estimates on the 
harmonic part of its Newtonian potential, but not on the Newtonian potential. We resolve this issue by using 
a refined Arzel\`{a}-Ascoli's lemma
combined with the uniform decay estimates which guarantee the convergence of interaction energy terms 
for all time.

Before concluding the introduction, we would like to refer the interested readers 
to \cite{BL13, Hybrid} for related articles 
in the context of parabolic-parabolic PKS system of singles populations and to
\cite{two1,two2,two3,two4,two5} for the two component chemotactic systems.
For a profound discussions on the gradient flow in Wasserstein space and its application to a 
large class of PDEs, we refer to
\cite{Villani03,AGS,San}. 

\subsection{Main results} The main result of this article is the following:

\begin{Th}\label{main}
Assume $\bm{\rho}^0 \in \g_2$ and $\bm{\beta}$ is sub-critical. Then the PKS-system \eqref{kss} admits 
a global weak solution $\bm{\rho}$ in the sense of definition \ref{weak sol} with initial data $\bm{\rho}^0.$
Moreover, $\bm{\rho}$ satisfies for every $T>0$
\begin{itemize}
  \item[(a)] $\bm{\rho} \in \left(L^2((0,T)\times \rt)\right)^n \cap \left(L^1(0,T;W^{1,1}(\rt))\right)^n,$ and 
  Fischer information bound:   
 \begin{align*}
 \si\int_{0}^{T}  \inte \left|\frac{\nabla \rho_{i}(x,t)}{\rho_{i}(x,t)}\right|^2 \rho_{i}(x,t) \ dxdt < +\infty.
\end{align*}
\item[(b)] Free energy inequality:
\begin{align*} 
\si\int_{0}^{T}  \inte \left|\frac{\nabla \rho_{i}(x,t)}{\rho_{i}(x,t)} 
 - \sj a_{ij}\nabla u_{j}(x,t)\right|^2 \rho_{i}(x,t) \ dxdt + \f(\bm{\rho}(T)) \leq \f(\bm{\rho}^0).
\end{align*}
\item[(c)] The weak solution obtained in (a) is unique. 
\end{itemize}
\end{Th}

We divide the article into the following sections:
In section $2$ we introduce a few notations used throughout
this article and recall a few known properties of the Wasserstein distance and the free energy 
functional $\f.$ In section $3$ we propose the time-discretized JKO-scheme (see \eqref{mms}) and 
prove its well-definedness. Section 4 is devoted to the a priori estimates and regularity estimates 
for the discrete time interpolates, which is a crucial step towards establishing the Euler Lagrange 
equation obtained in Section 5. In section 6 we define the time interpolation and prove the a priori 
estimates and regularity results as a byproduct of the results obtained in section 4. In section 7 we 
prove the convergence of the scheme and hence obtain the existence of a solution to the PKS-system \eqref{kss}.
Finally, in section $8,$ we show that the obtained solution satisfies the free energy inequality 
Theorem \ref{main}(b) and prove the uniqueness result. 

\vspace{0.2 cm}

\noindent
{\bf Acknowledgement:} D. Karmakar is partially supported by Technion fellowship grant.

\section{Notations and Preliminaries}
\subsection{Wasserstein distance}
In this section, we recall the definition of Wasserstein 2-distance 
(also called the Monge-Kantorovich distance of order $2$) and some of its well-known properties. 

Let $\mathcal{P}(\rt)$ be the space of all Borel probability measures on $\rt,$ 
$\mathcal{P}_{2}(\rt)$  denotes the subset of $\mathcal{P}(\rt)$ having finite second
moments and $\mathcal{P}_{ac,2}(\rt)$ denotes the subset of $\mathcal{P}_2(\rt)$ which are
absolutely continuous with respect to the Lebesgue measure on $\rt.$
 
Given two elements $\mu, \nu$ of $\mathcal{P}(\rt)$ and a map $T: \rt \rightarrow \rt,$ we say $T$ pushes
forward $\mu$ to $\nu,$ denoted by $T\#\mu = \nu,$ if for every Borel measurable subset $U$ of $\rt,$
$\nu(U) = \mu(T^{-1}(U)).$ Equivalently, 
\begin{align}\label{cov}
 \inte \psi(x) d\nu(x) = \inte \psi(T(x)) d\mu(x), \ \ for \ every \ \psi \in L^1(\rt,d\nu).
\end{align}

On $\mathcal{P}_2(\rt)$ we can define a distance $d_{\mbox{{\tiny{W}}}},$ using the Monge-Kantorovich transportation problem with quadratic
cost function $c(x,y) = |x-y|^2.$ More precisely, given $\mu, \nu \in \mathcal{P}_2(\rt)$ define
\begin{align}\label{mk}
 d_{\mbox{{\tiny{W}}}}^2(\mu,\nu) := \inf_{\pi \in \Pi(\mu,\nu)} \left[\inte \inte |x-y|^2 \ \pi(dxdy) \right],
\end{align}
where
\begin{align*}
\Pi(\mu,\nu):= \{\pi \in \mathcal{P}(\rt \times \rt)| \ (P_1) \# \pi= \mu, \ (P_2) \# \pi= \nu\},
\end{align*}
is the set of transport plans and $P_i: \rt \times \rt \rightarrow \rt$ denotes the canonical projections 
on the $i$-th factor.

The celebrated  theorem of Brenier \cite{Brenier} asserts that: if $\mu \in \mathcal{P}_{ac,2}(\rt)$ then there exists a unique 
(up to additive constants) convex, lower semi continuous function $\varphi$ such that $\nabla \varphi \# \mu = \nu$ and the optimal
transference plan $\hat \pi$ on the right hand side of \eqref{mk} is given by $\hat \pi = (Id, \nabla \varphi)
\# \mu,$ where $Id : \rt \rightarrow \rt$ is the identity mapping (see \cite[Theorem $2.12$]{Villani03}). 
As a consequence, we have 
\begin{align} \label{wd1}
 d_{\mbox{{\tiny{W}}}}^2(\mu, \nu) = \inte |x-\nabla \varphi(x)|^2 d\mu(x), \ \ \ where \ 
 \nabla \varphi \# \mu = \nu.
\end{align}

Note that if $\mu, \nu$ are two non-negative measures on $\rt$ (not necessarily probability measures) satisfying 
the total mass compatibility condition
$\mu(\rt) = \nu(\rt) (=\beta >0),$ then we can also define the Wasserstein $2$-distance between them as
follows:
\begin{align}\label{wd}
 d_{\mbox{{\tiny{W}}}}(\mu,\nu) = \beta^{\frac{1}{2}}d_{\mbox{{\tiny{W}}}}\left(\frac{\mu}{\beta}, \frac{\nu}{\beta}\right).
\end{align}
We will denote by $\mathcal{P}^{\beta}(\rt)$ the space of non-negative Borel measures with total mass $\beta$
and $\mathcal{P}^{\beta}_{2}(\rt)$ and $\mathcal{P}^{\beta}_{ac,2}(\rt)$ are defined analogously. We will also
use the bold $\bm{\beta}$ notation in $\mathcal{P}^{\bm{\beta}}(\rt)$ to denote the product space $\mathcal{P}^{\beta_1}(\rt)
\times \cdots \times \mathcal{P}^{\beta_n}(\rt).$

One advantage of defining the Wasserstein distance on $\mathcal{P}^{\beta}(\rt)$ by \eqref{wd} is that, 
if $\mu$ is absolutely continuous with respect 
to the Lebesgue measure and if $\nabla \varphi$ is the gradient of a convex function pushing $\mu/\beta$ forward to 
$\nu/\beta$ then $\nabla \varphi \# \mu = \nu$ and 
\begin{align*}
 d_{\mbox{{\tiny{W}}}}^2(\mu, \nu) = \inte |x - \nabla \varphi (x)|^2 d\mu(x),
\end{align*}
where note that $d_{\mbox{{\tiny{W}}}}(\mu, \nu)$ is defined by \eqref{wd}.

\subsection{Change of variable formula}
Let $\mu, \nu \in \mathcal{P}_{ac,2}(\rt)$ be two probability measures and let $f$ and $g$ be their respective densities. Let $\varphi$
be a convex function such that $\nabla \varphi \# fdx = gdx.$ $\varphi$ being convex, by Aleksandrov's theorem, 
it is twice differentiable almost everywhere in its domain of definition. Let $D^2\varphi$ denotes the Hessian 
matrix of $\varphi$ (which is well defined a.e. on $\{\varphi < +\infty\}$) and let $det_A D^2\varphi$ be the determinant of $D^2\varphi.$ 
Here we use the notation $det_A$ because of fact that $det_A D^2\varphi$ is the absolutely continuous part of the 
Hessian measure $det_H D^2\varphi.$ The following change of variable formula can be found in 
\cite{McCann} (see also Villani \cite[Theorem $4.8$]{Villani03}):

{\it
For every measurable function $U : [0,\infty) \rightarrow \mathbb{R}$ such that $U$ is bounded from below and $U(0)=0$
\begin{align*}
\inte U(g(x))dx = \inte U\left(\frac{f(x)}{det_A D^2 \varphi(x)}\right)det_A D^2 \varphi(x)dx.
\end{align*}
}
In particular, choosing $U(t) = t\ln t$ and $f,g$ any non-negative densities satisfying mass compatibility condition
$\inte f(x)dx = \inte g(x)dx = \beta$ with  $\nabla \varphi \# fdx = gdx$ then 
\begin{align} \label{cv2}
\inte g(x)\ln g(x) dx = \inte f(x)\ln f(x) dx - \inte f(x) \ln \left(det_A D^2 \varphi(x)\right) dx.
\end{align}

\subsection{Properties of the free energy functional}
We end this section by recalling a few well known properties of the free energy functional $\f$
 whose proof can be found in \cite{SW,KW}.
\begin{Prop}\label{propertyf}
The followings hold:
 \begin{itemize}
  \item[(a)] $\f$ is bounded from below on $\g$ if and only if $\bm{\beta}$ satisfies 
  \begin{align*}
   \Lambda_I(\bm{\beta}) = 0 \ \ and \ \eqref{beta}.
\end{align*}
\item[(b)] For any $n$-numbers $\alpha_i >0,$ the functional 
\begin{align*}
\f_{\bm{\alpha}}(\bm{\rho}):=\f(\bm{\rho}) + \si \alpha_iM_2(\rho_i) 
\end{align*}
is bounded from below on $\g_2$ if and only if $\bm{\beta}$ satisfies \eqref{beta}.
\item[(c)] The functionals $\f$ and $\f_{\bm{\alpha}}$ are sequentially lower semi-continuous with respect to the
weak topology of $L^1(\rt).$
\item[(d)] If $\bm{\beta}$ if sub-critical and $\min_{i \in I} \alpha_i>0,$ then all the sub-level 
sets $\{\f_{\bm{\alpha}} \leq C\}$ are sequentially precompact 
with respect to the weak topology of $L^1(\rt).$
\end{itemize}
\end{Prop}

\begin{Rem} \label{rem2}
If $\bm{\beta}$ satisfies \eqref{beta}, we denote $C_{\mbox{\tiny{LHLS}}}(\bm{\beta},\bm{\alpha})
= \inf_{\bm{\rho}\in \g_2} \f_{\bm{\alpha}}(\bm{\rho}).$ Obviously,
if $\bm{\beta}$ is not critical (i.e., $\Lambda_{I}(\bm{\beta}) > 0$) then $C_{\mbox{\tiny{LHLS}}}(\bm{\beta},\bm{\alpha})
\rightarrow -\infty$ as $\bm{\alpha} \searrow \bm{0}.$
\end{Rem}

\section{Minimizing Movement Scheme}
Given two elements $\bm{\rho}, \bm{\eta} \in \mathcal{P}^{\bm{\beta}}_2(\rt),$ 
we define the distance between them as follows
\begin{align*}
 \bm{d}_{\mbox{{\tiny{W}}}}(\bm{\rho}, \bm{\eta}) 
 = \left[\si d_{\mbox{{\tiny{W}}}}^2(\rho_i, \eta_i)\right]^{\frac{1}{2}},
\end{align*}
where $d_{\mbox{{\tiny{W}}}}$ is defined by \eqref{wd}.
\subsection{Minimizing Movement Scheme (MM-scheme):} Given $\bm{\rho}^0 \in \g_2$ and a time step $\tau
\in (0,1),$ set $\bm{\rho}^0_{\tau} = \bm{\rho}^0$ and define recursively
\begin{align}\label{mms}
 \bm{\rho}^{k}_{\tau} \in \arg \min_{\bm{\rho}\in \g_2} \left\{\f(\bm{\rho}) + \frac{1}{2\tau}\bm{d}_{\mbox{{\tiny{W}}}}^2(\bm{\rho},
 \bm{\rho}^{k-1}_{\tau})\right\}, \ \ k \in \mathbb{N}.
\end{align}

The following proposition vindicates that the MM-scheme \eqref{mms} is well-defined.
\begin{Prop} 
\begin{itemize}
\item[(a)] Assume $\bm{\beta}$ satisfies \eqref{beta} and fix $\bm{\eta} \in \g_2$ and $\tau>0.$ 
 Then the functional 
 $\mathcal{G}_{\bm{\eta}}$ defined by
 \begin{align*}
  \mathcal{G}_{\bm{\eta}}(\bm{\rho}) := \f(\bm{\rho}) + \frac{1}{2\tau}\bm{d}_{\mbox{{\tiny{W}}}}^2(\bm{\rho},\bm{\eta}) 
 \end{align*}
is bounded from below on $\g_2.$ Moreover, $\mathcal{G}_{\bm{\eta}}$ is sequentially lower semicontinuous with respect to the 
weak topology in $L^1(\rt).$
\item[(b)] If $\bm{\beta}$ is sub-critical then all the sub-level sets $\{\mathcal{G}_{\bm{\eta}} \leq C\}$ are 
sequentially precompact 
with respect to the weak topology in $L^1(\rt).$ In particular, the minimization problem $\inf_{\bm{\rho \in \g_2}}
\mathcal{G}_{\bm{\eta}}(\bm{\rho})$ admits a solution.
\end{itemize}
\end{Prop}

\begin{proof}
 \begin{itemize}
  \item[(a)] The proof of $(a)$ is a simple consequence of Proposition \ref{propertyf}(b) and the inequality
  \begin{align} \label{1}
   M_2(\bm{\rho}) \leq 2 \bm{d}_{\mbox{{\tiny{W}}}}^2(\bm{\rho}, \bm{\eta}) + 2M_2(\bm{\eta}).
  \end{align}
  which follows from 
 the triangle inequality.
As a consequence, we obtain $\mathcal{G}_{\bm{\eta}}(\bm{\rho}) \geq \f(\bm{\rho}) + \frac{1}{4\tau}M_2(\bm{\rho}) 
-\frac{1}{2\tau} M_2(\bm{\eta})$ is bounded from below on $\g_2.$ The 
sequentially lower semi-continuity of $\mathcal{G}_{\bm{\eta}}$ follows from the sequentially lower semi-continuity 
of $\f$ (see \cite{SW,KW}) and that of 
Wasserstein distance with respect to the weak topology of $L^1(\rt)$ (see \cite{JKO}).

\item[(b)] Any sub-level set $\{\mathcal{G}_{\bm{\eta}}\leq C\}$ 
is contained in the sub-level 
set $\{\f_{\bm{\frac{1}{4\tau}}} \leq  C + \frac{1}{2\tau}M_2(\bm{\eta})\}.$ Therefore, the conclusion of the 
proposition follows from Proposition \ref{propertyf}(d) together with the first part of the current
proposition.
\end{itemize}
\end{proof}

\begin{Rem}
The functional $\f$ is not convex, neither in the usual sense nor it is displacement convex
(in the sense of McCann \cite{McCann}).  As a consequence,  there may not be a unique minimizer 
in \eqref{mms}. We pick a minimizer recursively from the problem \eqref{mms}. Surprisingly, every 
choice gives rise to the same solution to \eqref{kss} in the limit $\tau \rightarrow 0$.  
Indeed, for the scalar case $n=1,$ Fern\'{a}ndez and Mischler \cite{FM} 
have obtained the uniqueness of the solutions  satisfying the free energy inequality, 
using an argument introduced by Ben-Artzi \cite{Artzi, ABrezis} for 2D viscous vortex model.
Their argument can be used to the system case and prove the uniqueness of solutions satisfying the
free energy inequality. 
\end{Rem}

\vspace{0.5 cm}

\section{A priori Estimates and the regularity of minimizers} 
\subsection{A priori estimates}

\begin{Lem} \label{apriori}
 For every $T \in (0,\infty)$ there exists a constant $\mathcal{C}_{ap}(T)$ 
 such that for each $\tau \in (0,1)$ and positive integers $k$ satisfying $k\tau \leq T$ there holds 
 \begin{align} \label{ap1}
  M_2(\bm{\rho}_{\tau}^k) + \frac{1}{\tau}\sum_{l=1}^{k} \bm{d}_{\mbox{{\tiny{W}}}}^2(\bm{\rho}_{\tau}^{l}, \bm{\rho}_{\tau}^{l-1})
 + \si \inte \rho_{\tau,i}^k|\ln \rho_{\tau,i}^k| \ dx\leq \mathcal{C}_{ap}(T).
 \end{align}
\end{Lem}
\begin{proof}
 For every $l \in \{1,\ldots,k\},$ the minimizing property of $\rl$ gives
 \begin{align*}
  \f(\rl) + \frac{1}{2\tau}\bm{d}_{\mbox{{\tiny{W}}}}^2(\rl, \bm{\rho}^{l-1}_{\tau}) \leq \f(\bm{\rho}^{l-1}_{\tau}). 
 \end{align*}
Summing over $l \in \{1,\ldots,k\}$ we obtain
\begin{align} \label{ap2}
 \f(\rk) + \frac{1}{2\tau}\sum_{l=1}^{k} \bm{d}_{\mbox{{\tiny{W}}}}^2(\bm{\rho}_{\tau}^{l}, \bm{\rho}_{\tau}^{l-1}) \leq \f(\bm{\rho}^0).
\end{align}
Choose $\alpha >0$ such that $8\alpha T < 1.$ Applying \eqref{1}, triangle inequality and Cauchy-Schwartz inequality
we get 
\begin{align} \label{ap3}
 \alpha M_2(\rk) &\leq 2\alpha\bm{d}_{\mbox{{\tiny{W}}}}^2(\rk,\bm{\rho}^0) + 2\alpha M_2(\bm{\rho}^0) 
 \leq 2\alpha \left(\sum_{l=1}^k \bm{d}_{\mbox{{\tiny{W}}}}(\rl, \bm{\rho}^{l-1}_{\tau})\right)^2 + 2\alpha M_2(\bm{\rho}^0) \notag \\
 &\leq 2\alpha k \sum_{l=1}^k \bm{d}_{\mbox{{\tiny{W}}}}^2(\rl, \bm{\rho}^{l-1}_{\tau}) + 2\alpha M_2(\bm{\rho}^0) 
 \leq \frac{2\alpha T}{\tau} \sum_{l=1}^k \bm{d}_{\mbox{{\tiny{W}}}}^2(\rl, \bm{\rho}^{l-1}_{\tau}) + 2\alpha M_2(\bm{\rho}^0) \notag \\ 
 &\leq \frac{1}{4\tau} \sum_{l=1}^k \bm{d}_{\mbox{{\tiny{W}}}}^2(\rl, \bm{\rho}^{l-1}_{\tau}) + 2\alpha M_2(\bm{\rho}^0).
\end{align}
Adding $\alpha M_2(\rk)$ on both sides of \eqref{ap2}, using \eqref{ap3} and Proposition \ref{propertyf} (b)
(and Remark \ref{rem2}) we obtain
\begin{align*}
 \f(\bm{\rho}^0) + \frac{1}{4\tau} \sum_{l=1}^k \bm{d}_{\mbox{{\tiny{W}}}}^2(\rl, \bm{\rho}^{l-1}_{\tau}) + 2\alpha M_2(\bm{\rho}^0)
 &\geq \f(\rk) + \alpha M_2(\rk) + \frac{1}{2\tau}\sum_{l=1}^{k} \bm{d}_{\mbox{{\tiny{W}}}}^2(\bm{\rho}_{\tau}^{l}, \bm{\rho}_{\tau}^{l-1}) \\
 &\geq C_{\mbox{\tiny{LHLS}}}(\bm{\beta},\bm{\alpha}) + \frac{1}{2\tau}\sum_{l=1}^{k} \bm{d}_{\mbox{{\tiny{W}}}}^2(\bm{\rho}_{\tau}^{l}, \bm{\rho}_{\tau}^{l-1}),
\end{align*}
which in turn gives 
\begin{align} \label{ap4}
 \frac{1}{\tau}\sum_{l=1}^{k} \bm{d}_{\mbox{{\tiny{W}}}}^2(\bm{\rho}_{\tau}^{l}, \bm{\rho}_{\tau}^{l-1}) \leq 
 4(\f(\bm{\rho}^0) + 2\alpha M(\bm{\rho}^0) - C_{\mbox{\tiny{LHLS}}}(\bm{\beta}, \bm{\alpha})) =: C_1(T).
\end{align}
Again using \eqref{ap4} in \eqref{ap3} we get
\begin{align} \label{ap5}
 M(\rk) \leq \frac{1}{\alpha}(\f(\bm{\rho}^0) + 4\alpha M(\bm{\rho}^0) - C_{\mbox{{\tiny{LHLS}}}}(\bm{\beta}, \bm{\alpha}))=:C_2(T).
\end{align}
Finally, since $\f(\rk) \leq \f(\bm{\rho}^0)$ because of \eqref{ap2} and 
$\bm{\beta}$ is sub-critical  we obtain an upper bound on the entropy 
(see \cite[Theorem $4.1$]{KW})
\begin{align} \label{ap6}
 \si \inte \rho_{\tau,i}^k \ln \rho_{\tau,i}^k \leq C_3.
\end{align}
Hence by Lemma \ref{ce}(b) (see appendix) and the bounds \eqref{ap5}, \eqref{ap6} we get  
\begin{align} \label{ap7}
 \si \inte \rho_{\tau,i}^k |\ln \rho_{\tau,i}^k| \leq C_4(T).
\end{align}
We conclude the proof with 
$\mathcal{C}_{ap}(T) = C_1(T) + C_2(T) + C_4(T).$
\end{proof}

\subsection{Regularity of minimizers}
We indicated in the introduction that in order to derive the Euler-Lagrange equation satisfied by 
the minimizers in MM-scheme, we need 
additional regularities on the minimizers. This is the goal of this subsection. 
We will utilize the flow interchange technique introduced by 
Matthes-McCann and Savar\'{e} \cite{MMS}. Before stating our regularity results
let us first review this technique very briefly in our setting.

\subsubsection{Matthes-McCann-Savar\'{e} flow interchange technique}
Let \ $\mathcal{N} : \mathcal{P}^{\bm{\beta}}(\rt) \rightarrow \mathbb{R}\cup \{+\infty\}$ be a proper lower semi-continuous functional
and let $\mathcal{D}(\mathcal{N}) :=\{\bm{\mu} \in \mathcal{P}^{\bm{\beta}}(\rt)| \ \mathcal{N}(\bm{\mu}) 
< +\infty\}$ be the domain of $\mathcal{N}.$
Further assume that $\mathcal{N}$ generates a continuous semigroup $\mathcal{S}^{\mathcal{N}}_t :
\mathcal{D}(\mathcal{N}) \rightarrow \mathcal{D}(\mathcal{N}),$ 
satisfying the evolution variational identity {\bf(EVI)}
\begin{align}\label{evi}
 \frac{1}{2} \frac{\bm{d}_{\mbox{{\tiny{W}}}}^2(\mathcal{S}^{\mathcal{N}}_t(\bm{\mu}), \bm{\nu}) - \bm{d}_{\mbox{{\tiny{W}}}}^2(\bm{\mu},\bm{\nu})}{t}
\leq \mathcal{N}(\bm{\nu}) - \mathcal{N}(\mathcal{S}^{\mathcal{N}}_t(\bm{\mu})), \ \ \forall \ \bm{\mu}, \bm{\nu}
\in \mathcal{D}(\mathcal{N}).
\end{align}
Assume for simplicity $\g_2 \subset \mathcal{D}(\mathcal{N})$ and $\g_2$ is an invariant subset under the 
flow $\mathcal{S}^{\mathcal{N}}_t$ i.e., $\mathcal{S}^{\mathcal{N}}_t(\g_2)
\subset \g_2.$ If $\bm{\mu} \in \g_2$ we define
the  dissipation of $\f$ along the flow $(\mathcal{S}^{\mathcal{N}}_t)_{t \geq 0}$ of $\mathcal{N}$ by
\begin{align*}
 \bm{\nabla}^{\mathcal{N}}\f(\bm{\mu}):=\limsup_{t\searrow 0} \frac{\f(\bm{\mu}) - \f(\mathcal{S}^{\mathcal{N}}_t(\bm{\mu}))}{t}.
\end{align*}
By the minimizing property of $\rk$  (see \eqref{mms})
\begin{align} \label{1a}
 \f(\rk) + \frac{1}{2\tau}\bm{d}_{\mbox{{\tiny{W}}}}^2(\rk, \rka) \leq \f(\bm{\rho}) + \frac{1}{2\tau}\bm{d}_{\mbox{{\tiny{W}}}}^2(\bm{\rho}, \rka),
\end{align}
for all $\bm{\rho} \in \g_2.$ Now setting $\bm{\rho} = \mathcal{S}^{\mathcal{N}}_t(\rk)$ in \eqref{1a} for $t>0,$
dividing by $t$ and using \eqref{evi} we get
\begin{align} \label{aq1}
 \frac{\f(\rk) - \f(\mathcal{S}^{\mathcal{N}}_t(\rk))}{t} \leq \frac{\mathcal{N}(\rka) - \mathcal{N}(
 \mathcal{S}^{\mathcal{N}}_t(\rk))}{\tau}.
\end{align}
Since $\mathcal{N}$ is lower semi-continuous, passing to the limit as $t \searrow 0$ we obtain
\begin{align*}
 \bm{\nabla}^{\mathcal{N}}\f(\rk) \leq \frac{\mathcal{N}(\rka) - \mathcal{N}(\rk)}{\tau}.
\end{align*}
In addition, if the quantity $\bm{\nabla}^{\mathcal{N}}\f \geq 0,$ or, at least behaves nicely  
then we could possibly get adequate estimates on the minimizers $\rk.$

\begin{Rem}
 The dissipation of $\f$ along the flow $\mathcal{S}^{\mathcal{N}}_t$ was originally 
 denoted by $\bm{D}^{\mathcal{N}}$ in the article \cite{MMS}. In order to avoid confusion 
 with the domain of $\mathcal{N}$ we choose to denote it 
 by $\bm{\nabla}^{\mathcal{N}}.$
\end{Rem}

In our case we use the entropy functional 
\begin{align*}
 \mathcal{H}(\bm{\rho}) := \si \mathcal{H}(\rho_i) = \si \inte \rho_i \ln \rho_i \ dx,
\end{align*}
if $\bm{\rho}$ is absolutely continuous with respect to the Lebesgue measure and $+\infty$ everywhere else.
The functional $\mathcal{H}$ is a particular example of the class known as {\it displacement convex entropy}, 
which guarantees the existence of a continuous semigroup \cite{AGS}. 
However,  we don't have $\bm{\nabla}^{\mathcal{H}}\f \geq 0$, but we can control the negative part of 
$\bm{\nabla}^{\mathcal{H}}\f$  and achieve higher regularity of the minimizers.
Below we recall some few well-known facts on the 
{\it displacement convex entropy} in $\rt.$

\subsubsection{Displacement convex entropy}
Let $U: [0,+\infty) \rightarrow \mathbb{R}$ be a convex function satisfying
\begin{itemize}
 \item $U(0)=0,$ $U$ is continuous at $0$ and $U \in C^1(0,\infty),$
 \item $\lim_{t\rightarrow +\infty}\frac{U(t)}{t} = +\infty$ and $\lim_{t \rightarrow 0+} \frac{U(t)}{t^{\alpha}} > -\infty$
 for some $\alpha > \frac{1}{2},$
 \item $t\mapsto t^2U(t^{-2})$ is convex and decreasing in $(0,\infty).$
\end{itemize}
Define the functional $\mathcal{U}: \mathcal{P}_2(\rt) \rightarrow \mathbb{R} \cup \{+\infty\}$ by
\begin{align*}
 \mathcal{U}(\mu) = 
 \begin{cases}
  \inte U(\rho) \ dx, \ \ if \ \mu \in \mathcal{P}_{ac,2}(\rt), \ \mu = \rho dx, \\
  +\infty, \ \ \ \ \ \ \ \ \ \  \ \ elsewhere.
 \end{cases}
\end{align*}
The domain of $\mathcal{U}$ denoted by $\mathcal{D}(\mathcal{U})$ is the set of all densities $\rho$ such that $\mathcal{U}(\rho) < +\infty.$
Such a functional is called a displacement convex entropy with density function $U.$ 
It is well known that a displacement convex 
entropy $\mathcal{U}$ generates a continuous semigroup $\mathcal{S}^{\mathcal{U}}_t : \mathcal{D}(\mathcal{U})
\rightarrow \mathcal{D}(\mathcal{U})$ satisfying the Evolution Variational Identity
 \begin{align*}
  \frac{1}{2} \frac{d_{\mbox{{\tiny{W}}}}^2(\mathcal{S}^{\mathcal{U}}_t(\rho), \eta) - d_{\mbox{{\tiny{W}}}}^2(\rho,\eta)}{t}
\leq \mathcal{U}(\eta) - \mathcal{U}(\mathcal{S}^{\mathcal{U}}_t(\rho)), \ \ \forall \ \, \rho, \eta
\in \mathcal{D}(\mathcal{U}),
 \end{align*}
and $w := \mathcal{S}^{\mathcal{U}}_t(\rho)$ is the unique distributional solution to the Cauchy problem 
\begin{align*}
 \partial_t w = \Delta (L_U(w)), \ \ \ w(0) = \rho,
\end{align*}
where $L_U(t) = tU^{\prime}(t) - U(t).$ We refer to \cite[Theorem $11.2.5$]{AGS} for a proof of the 
claims described above.

It is easy to see that $H(t) := t\ln t$ satisfies above all criterion of a displacement convex entropy. 
Also, note that for our purpose we need to apply the above results to the densities $\rho$ having mass $\beta>0.$ 
Let $\rho, \eta$ be two densities with mass $\beta$ having finite entropies
and finite second moments.
Then applying EVI to the normalized densities gives 
\begin{align*}
  \frac{1}{2} \frac{d_{\mbox{{\tiny{W}}}}^2(\mathcal{S}^{\mathcal{H}}_t(\rho), \eta) - d_{\mbox{{\tiny{W}}}}^2(\rho,\eta)}{t}
  &= \beta \ \frac{1}{2} \frac{d_{\mbox{{\tiny{W}}}}^2(\mathcal{S}^{\mathcal{H}}_t(\frac{\rho}{\beta}), \frac{\eta}{\beta}) 
  - d_{\mbox{{\tiny{W}}}}^2(\frac{\rho}{\beta},\frac{\eta}{\beta})}{t}\\
&\leq \beta\left[\mathcal{H}\left(\frac{\eta}{\beta}\right) - \mathcal{H}\left(\mathcal{S}^{\mathcal{H}}_t
\left(\frac{\rho}{\beta}\right)\right)\right]
=\mathcal{H}(\eta) - \mathcal{H}(\mathcal{S}^{\mathcal{H}}_t(\rho)),
\end{align*}
since $L_{H}(t) = t,$ and $w := \mathcal{S}^{\mathcal{H}}_t(\rho)$ is a distributional solution to 
\begin{align*}
 \partial_t w = \Delta w, \ \ \ w(0) = \rho.
\end{align*}
Moreover, note that  $M_2(\mathcal{S}^{\mathcal{H}}_t(\rho)) < +\infty$
for all $t$ since $M_2(\rho) < +\infty.$
\subsubsection{Regularity results}
\begin{Lem}\label{regularity}
Let $\tau \in (0,1)$ and let $\rk$ be a sequence obtained using the MM-scheme \eqref{mms} 
satisfying
\begin{align}\label{theta}
 \si\inte \rki |\ln \rki| \ dx + M_2(\rk) \leq \Theta
\end{align}
for some constant $\Theta>0.$ Then $\rki \in W^{1,1}(\rt), \frac{\nabla \rki}{\rki} \in L^2(\rt,\rki)$ for all 
$i \in I.$

Moreover, there exists a constant $C(\Theta)$ such that 
\begin{align} \label{w11 estimate}
 \si \inte \left|\frac{\nabla \rki(x)}{\rki(x)}\right|^2 \rki(x) \ dx \leq \frac{2}{\tau}
 \left[\mathcal{H}(\rka) - \mathcal{H}(\rk)\right] + C(\Theta).
\end{align}
\end{Lem}

\begin{proof}
 For each $i \in \{1,\ldots,n\},$ let $\mathcal{S}_t^{\mathcal{H}} : \mathcal{D}(\mathcal{H}) \rightarrow 
 \mathcal{D}(\mathcal{H})$ be the continuous semigroup generated by $\mathcal{H}$ with respect to the 
 Wasserstein distance in $\mathcal{P}^{\beta_i}_2(\rt)$. For simplicity 
 of notations we define $w_{t,i} = \mathcal{S}^{\mathcal{H}}_t(\rki)$ and $\bm{w}_t = 
 (w_{t,1},\ldots, w_{t,n}).$ Then $w_{t,i}$ satisfies
 \begin{align*}
  \partial_t w_{t,i} = \Delta w_{t,i}, \ \ w_{0,i} = \rki, \ \ for \ all \ i=1,\ldots,n.
 \end{align*}
Moreover, since $\rk \in \g_2$ we get the following outcomes: for all $i \in \{1,\ldots,n\}$
\begin{itemize}
 \item[(a)] for each $T>0$ there exists a constant $C(T)>0$ such that $M_2(w_{t,i}) \leq C(T),$ 
 for all $t \leq T,$
 \item[(b)] $\lim_{t \rightarrow 0+}||w_{t,i} - \rki||_{L^1(\rt)} = 0,$
 \item[(c)] $\lim_{t\rightarrow 0+}\inte w_{t,i} \ln w_{t,i} \ dx = \inte \rki \ln \rki \ dx.$
 \item[(d)] $\frac{d}{dt} \inte w_{t,i} \ln w_{t,i} \ dx = 
 - \inte \left|\frac{\nabla w_{t,i}}{w_{t,i}}\right|^2 w_{t,i} \ dx <0,$ for $t>0.$
\end{itemize}
  For the convenience of the reader we prove  $(c)$.
  Since $H$ is convex by strong $L^1$-convergence $(b)$ ($L^1$-weak convergence is enough though) we have 
 \begin{align*}
  \inte \rki \ln \rki \ dx \leq \liminf_{t \rightarrow 0+} \inte w_{t,i} \ln w_{t,i} \ dx.
 \end{align*}
 On the other hand, by uniqueness of the solutions to the heat equation, $w_{t,i}$ is the convolution (with
 respect to the space variable) of 
 $\rki$ with the heat kernel $G_t(x) :=\frac{1}{\sqrt{4\pi t}}e^{-\frac{|x|^2}{4t}}.$ Again by convexity of $H$
 and Jensen's inequality
 \begin{align*}
 H(w_{t,i})(x) = H \left(\inte \rki(x-y)G_t(y) \ dy \right) \leq \inte H(\rki)(x-y)G_t(y) \ dy.
 \end{align*}
Integrating with respect to $x$ and taking $\limsup$ as $t \rightarrow 0+$ we obtain
 \begin{align*}
  \limsup_{t \rightarrow 0+} \inte w_{t,i} \ln w_{t,i} \ dx \leq \inte \rki \ln \rki \ dx.
 \end{align*}

\noindent
{\bf Step 1.} $t\mapsto \f(\bm{w}_t)$ is continuous at $0$ and differentiable in $(0,\infty).$

\vspace{0.1 cm}

\noindent
{\it Proof of step 1.} The differentiability property of the functional follows from the smoothness of $w_{t}$ as
it is being the convolution of $\rki$ with the heat kernel 
$G_t(x).$ So, we only need to check the continuity at $0.$ In fact, by $(a),(b)$ and $(c)$
we can argue as in \cite[Lemma $3.1$]{JKO2} (see also \cite{KW}) to obtain
\begin{align*}
 \lim_{t\rightarrow 0+} \inte \inte w_{t,i}(x)\ln|x-y|w_{t,j}(y)\ dxdy = \inte \inte \rki(x)\ln|x-y|
 \rho_{\tau,j}^{k}(y) \ dxdy,
\end{align*}
and hence $\lim_{t\rightarrow 0+} \f(\bm{w}_t) = \f(\rk).$

Define $u_{t,i}(x) = -\frac{1}{2\pi}\inte \ln|x-y|w_{t,i}(y)\ dy,$ then $-\Delta u_{t,i} = w_{t,i}.$
Owing to the smoothness of $\bm{w}_t$ and using the symmetry of $(a_{ij})$ we get for $t>0$
\begin{align}\label{ap8}
 \frac{d}{dt}\f(\bm{w}_t) = &\si\inte (1 + \ln w_{t,i})\partial_t w_{t,i}\ dx \notag\\
 &+ \si\sj \frac{a_{ij}}{2\pi} \inte \inte \partial_t w_{t,i}(x)\ln|x-y| w_{t,j}(y) \ dxdy \notag\\
 =& - \si \inte \nabla (\ln w_{t,i})\cdot \nabla w_{t,i} + \si \sj a_{ij}\inte \nabla w_{t,i}\cdot \nabla u_{t,j} \notag \\
 =& - \si \inte \left|\frac{\nabla w_{t,i}}{w_{t,i}}\right|^2w_{t,i} \ dx + \si\sj a_{ij} \inte w_{t,i}w_{t,j} \ dx. 
\end{align}

\noindent
{\bf Step 2.} There exists a constant $C_1(\Theta)$  such that 
\begin{align}\label{ap9}
 \frac{d}{dt}\f(\bm{w}_t) \leq -\frac{1}{2} \si \inte \left|\frac{\nabla w_{t,i}}{w_{t,i}}\right|^2w_{t,i} \ dx + C_1(\Theta),
 \ \ for \ all \ t \in (0,1).
\end{align}
\noindent
{\it Proof of step 2.} By Cauchy-Schwartz inequality, Lemma \ref{ce}(a) (see appendix) and $(d)$
\begin{align*}
 \si\sj a_{ij} \inte w_{t,i}w_{t,j} \ dx &\leq C_2\si \inte (w_{t,i})^2 \ dx \\
 &\leq C_2 \left[\epsilon\si \left|\left| \frac{\nabla w_{t,i}}{w_{t,i}}\right|\right|_{L^2(\rt,w_{t,i})}^2 ||
   w_{t,i}\ln w_{t,i}||_{L^1(\rt)} + L_{\epsilon}\right] \\
   &\leq C_3(\Theta)\epsilon \si \inte \left|\frac{\nabla w_{t,i}}{w_{t,i}}\right|^2w_{t,i} \ dx + C_4(\epsilon).
\end{align*}
Choosing $\epsilon >0$ small such that $C_3(\Theta)\epsilon < \frac{1}{2}$ and using \eqref{ap8}
we obtain \eqref{ap9}.

\vspace{0.1 cm}

\noindent
{\bf Step 3:} $L^2$-regularity of $\rk.$

\vspace{0.1 cm}

\noindent
{\it Proof of step 3.}
Let $t(m)$ be a sequence such that $t(m)\searrow 0$ as $m\rightarrow \infty.$ Since $t\mapsto \f(w_t)$ is continuous on $[0,\infty)$
and differentiable in $(0,\infty),$ by Lagrange mean value theorem there exists $r(m)\in (0, t(m))$ such that 
\begin{align} \label{ap10}
 \frac{\f(\rk) - \f(\bm{w}_{t(m)})}{t(m)} = -\frac{d}{dt}\f(\bm{w}_t)\Big|_{t=r(m)}.
\end{align}

Combining \eqref{aq1} and \eqref{ap10} we get
\begin{align} \label{ab} 
-\frac{d}{dt}\f(\bm{w}_t)\Big|_{t=r(m)} = \frac{\f(\rk) - \f(\bm{w}_{t(m)})}{t(m)} 
\leq \frac{\mathcal{H}(\rka) - \mathcal{H}(\bm{w}_{t(m)})}{\tau}
\end{align}
Now Step 2 and \eqref{ab} gives
\begin{align} \label{ap11}
 \si \inte \left|\frac{\nabla w_{r(m),i}}{w_{r(m),i}}\right|^2w_{r(m),i} \ dx \leq \frac{2}{\tau}
 \left[\mathcal{H}(\rka) - \mathcal{H}(\bm{w}_{t(m)})\right] + 2C_1(\Theta).
\end{align}
Since $\mathcal{H}(\bm{w}_t)$ remains bounded for small time (thanks to $(d)$) we get from \eqref{ap11}
\begin{align}\label{ap12}
 \si \inte \left|\frac{\nabla w_{r(m),i}}{w_{r(m),i}}\right|^2w_{r(m),i} \ dx \leq C_5(\Theta)\left(1 + \frac{1}{\tau}\right).
\end{align}
Using Lemma \ref{ce}$(a)$ again, we obtain
\begin{align*}
 \si\inte (w_{r(m),i})^2 \ dx\leq \frac{C_6(\Theta)}{\tau}.
\end{align*}
Since  $L^2(\rt)$ is a reflexive Banach space then, up to a subsequence (not labeled) we get  $w_{r(m),i} \rightharpoonup w_i$ weakly 
 in $L^2(\rt)$ as $m\rightarrow \infty.$ But we 
already know that $w_{r(m),i} \rightarrow \rki$ strongly in $L^1(\rt)$ as $m\rightarrow \infty.$
Hence, we must have $\rki =w_i,$ proving $\rki \in L^2(\rt).$

\vspace{0.1 cm}

\noindent
{\bf Step 4:} $W^{1,1}$-regularity of $\rk$ and the proof of \eqref{w11 estimate}.

\vspace{0.1 cm}

\noindent
{\it Proof of step 4.}
Define $v_{m,i} := \frac{\nabla w_{r(m),i}}{w_{r(m),i}}.$ Since $w_{r(m),i}$ narrowly converges to $\rki$ 
as $m\rightarrow \infty$
and \eqref{ap12} holds, we can invoke Proposition \ref{cvf} to conclude that there exists a vector field
$v_i \in L^2(\rt,\rki;\rt)$ such that 
\begin{align}\label{ap13}
 \inte \zeta \cdot v_{m,i}w_{r(m),i} \ dx \rightarrow \inte \zeta \cdot v_i\rki \ dx, \ \ as \ m\rightarrow \infty,
\end{align}
for all $\zeta \in C_c^{\infty}(\rt;\rt)$ and $i=1,\ldots,n$ and moreover,
\begin{align}\label{ap14}
 \inte |v_i|^2\rki \ dx\leq \liminf_{m\rightarrow \infty} \inte \left|\frac{\nabla w_{r(m),i}}{w_{r(m),i}}\right|^2w_{r(m),i} \ dx
 \leq C_5(\Theta)\left(1 + \frac{1}{\tau}\right).
\end{align}
Now since $v_{i}\rki \in L^1(\rt) \times L^1(\rt)$ and
\begin{align}\label{ap15}
 \inte \zeta\cdot v_{m,i}w_{r(m),i} = \inte \zeta \cdot \nabla w_{r(m),i} = -\inte (\nabla \cdot \zeta)w_{r(m),i}
 \rightarrow -\inte (\nabla \cdot \zeta)\rki
\end{align}
 we conclude from \eqref{ap13} and \eqref{ap15} that $\nabla \rki = v_i \rki.$ Which proves $\rki \in 
W^{1,1}(\rt).$ Finally from \eqref{ap14}, \eqref{ap11} and $(c)$ we get
\begin{align*}
 \inte \left|\frac{\nabla \rki}{\rki}\right|^2\rki \ dx \leq 
 &\liminf_{m\rightarrow \infty} \si \inte \left|\frac{\nabla w_{r(m),i}}{w_{r(m),i}}\right|^2w_{r(m),i} \ dx  \\
 &\leq \liminf_{m\rightarrow\infty} \frac{2}{\tau} \left[\mathcal{H}(\rka) - \mathcal{H}(\bm{w}_{t(m)})\right] + 2C_1(\Theta). \\
 &= \frac{2}{\tau} \left[\mathcal{H}(\rka) - \mathcal{H}(\rk)\right] + 2C_1(\Theta).
\end{align*}
This completes the proof of the lemma.
\end{proof}

Next we obtain regularity estimates on the Newtonian potential of $\rki.$
Recall that the Newtonian potential of $\rki$ is defined by 
\begin{align}\label{np}
 u_{\tau,i}^k(x) = -\frac{1}{2\pi}\inte \ln|x-y| \rki(y)dy, \ \ for \ i=1,\ldots,n.
\end{align}

\begin{Lem}\label{regularity u}
 Let $\tau \in (0,1)$ be given and assume that $\rk$ satisfies \eqref{theta}. Let $u_{\tau,i}^k$ be the Newtonian potential 
 defined by \eqref{np}. Then for $i \in \{1,\ldots,n\}$
  \begin{itemize}
  \item[(a)]  there exists a constant $C_{\Theta} >0$ and $R_0 >0$ such that 
  \begin{align*}
  \left|u_{\tau,i}^k(x) + \frac{\beta_i}{2\pi} \ln|x|\right| \leq C_{\Theta} 
  \ \ for \ all  \ x \in \rt \backslash B_R
  \ and \ R \geq R_0.
  \end{align*}
  \item[(b)] For each $R>0,$ there exists a constant $ C_{R,\Theta}>0$ such that 
  \begin{align*}
  ||u_{\tau,i}^k||_{H^1(B_R)} \leq C_{R,\Theta},\ \  \ 
 ||u_{\tau,i}^k||_{H^2(B_R)}^2 \leq \frac{2}{\tau}\left[\mathcal{H}(\rka) 
 - \mathcal{H}(\rk)\right] + C_{R,\Theta}.
  \end{align*}
 \end{itemize}
\end{Lem}

\begin{proof}
 $(a).$ Follows from the results of Chen-Li \cite{CL93} and by the hypothesis \eqref{theta},
 see \cite[Lemma $3.5$]{KW} for details. 

 $(b).$ For the first part, note that by $(a),$ $|u_{\tau,i}^k(x)| \leq C_1(\Theta, R)$ for 
 $x \in \partial B_{3R}$ and for some constant $C_1(\Theta, R)>0$ depending only on $\Theta$ and $R.$ Define 
 $v_{i}^k$ as follows
 \begin{align}\label{h1}
  \begin{cases}
  -\Delta v_i^k(x) = \rki(x) \ \ in \ B_{4R}, \\
  \ \ \ \ \ \ \ \ \ v_i^k = 0, \ \ \ \ \ \ \ on \ \partial B_{4R}.
  \end{cases}
\end{align}
Then by assumption \eqref{theta} and invoking the results of \cite{Stp,AF} we obtain
$v_i^k \in H^1_0(B_{4R}) \cap L^{\infty}(B_{4R})$ and $||v_i^k||_{L^{\infty}(B_{4R})} 
+ ||v_i^k||_{H^1(B_{4R})} \leq C_2(\Theta,R).$ Now,
$w_i^k := u_{\tau,i}^k - v_i^k$ is Harmonic in $B_{4R}$ and satisfies $|w_i^k(x)| \leq 
C_3(\Theta, R)$ for all $x \in \partial B_{3R}.$
By maximum principle, $|w_i^k(x)| \leq C_3(\Theta, R),$ for all $x \in B_{3R}.$ Furthermore, since $w_i^k$ is harmonic 
\begin{align} \label{h2}
 |D^{\alpha}w_i^k(x)| \leq \frac{C}{R^{2+|\alpha|}}||w_i^k||_{L^1(B_{3R})} \leq C_4(\Theta, R), 
 \ \ for \ all \ x \in B_R,
\end{align}
where $\alpha$ is a multi-index and $|\alpha|$ denotes the sum of all its components.
As a result, $||w_i^k||_{H^1(B_R)} \leq C_5(\Theta,R)$ and hence
\begin{align*}
 ||u_{\tau,i}^k||_{H^1(B_R)}  \leq ||v_i^k||_{H^1(B_R)} + ||w_i^k||_{H^1(B_R)} \leq C_6(\Theta,R).
\end{align*} 
For the second part, on one hand, by Lemma \ref{ce}(a), Lemma \ref{regularity} and by 
assumption on $\rki$ (equation \eqref{theta}) 
\begin{align*}
 ||\rki||_{L^2(\rt)}^2 &\leq \epsilon\left|\left|\frac{\nabla \rki}{\rki}\right|\right|_{L^2(\rt,\rki)}^2
 ||\rki\ln \rki||_{L^1(\rt)} + C(\epsilon) \\
 &\leq \Theta\epsilon\left|\left|\frac{\nabla \rki}{\rki}\right|\right|_{L^2(\rt,\rki)}^2 + C(\epsilon)
 \leq \frac{2\Theta\epsilon}{\tau}\left[\mathcal{H}(\rka) - \mathcal{H}(\rk)\right] + C(\epsilon, \Theta).
\end{align*}
So by elliptic regularity, $v_i^k$ defined in \eqref{h1} satisfies
\begin{align*}
 ||v_i^k||_{H^2(B_R)}^2 \leq C(R)||\rki||_{L^2(\rt)}^2 \leq
 \frac{\epsilon C_7(\Theta,R)}{\tau}\left[\mathcal{H}(\rka) - \mathcal{H}(\rk)\right] + C(\epsilon,\Theta, R).
\end{align*}
On the other hand by derivative estimate \eqref{h2}, $||w_i^k||_{H^2(B_R)} \leq C_8(\Theta, R).$ 
Choosing appropriate $\epsilon,$ we conclude the proof.
\end{proof}

\section{The Euler Lagrange equation}
\begin{Lem}\label{el}
 Let $\bm{\beta}$ be sub-critical, $\bm{\rho}^0 \in \g_2$ and $\tau>0$ be given. Let $\rk$ be the sequence 
 obtained iteratively by the MM-scheme \eqref{mms}. If $\nabla \varphi^k_i$ denotes the map transporting $\rkai$
 to $\rki$ then
 \begin{itemize}
  \item[(a)] Euler-Lagrange equation: 
 \begin{align}\label{el1}
  \frac{1}{\tau}\inte (\nabla \varphi^k_i(x)-x) \cdot  \zeta(\nabla\varphi^k_i(x))\rkai(x)dx
   =& -\inte  \zeta(x) \cdot\nabla\rki(x)dx \notag\\ 
   &+ \sj a_{ij} \inte  \zeta(x) \cdot \nabla u^{k}_{\tau,j}(x)\rki(x)dx.
 \end{align}
holds for all $i=1,\ldots, n$ and any $\zeta \in C_c^{\infty}(\rt;\rt).$ 
\item[(b)] Free energy production term:
the following identity holds
\begin{align}\label{el2}
 \frac{1}{\tau^2}d_{\mbox{{\tiny{W}}}}^2(\rki, \rkai) = 
 \inte \left|\frac{\nabla \rki(x)}{\rki(x)} - \sj a_{ij}\nabla u^k_{\tau,j}(x)\right|^2 \rki(x) dx.
\end{align}

\item[(c)] The following approximate weak solution is satisfied
\begin{align} \label{elfinal}
 &\left|\inte \psi(x)\left(\rki(x) - \rkai(x)\right) + \tau\inte  \nabla\psi(x) \cdot \nabla \rki \right. \notag\\
 &\left. -\tau\sj a_{ij} \inte  \nabla \psi(x) \cdot \nabla u^{k}_{\tau,j}(x)\rki(x)dx \right| 
 = O\left(||D^2\psi||_{L^{\infty}}\right) d_{\mbox{{\tiny{W}}}}^2(\rki,\rkai),
\end{align}
for all $\psi \in C_c^{\infty}(\rt).$
\end{itemize}
\end{Lem}

\begin{proof}
(a) Let $\zeta_i \in C_c^{\infty}(\rt;\rt), i = 1,\ldots,n$ be $n$-smooth vector fields.  For $\epsilon >0$ and for each 
 $i$ define $T_{\epsilon,i} :=x + \epsilon \zeta_i.$ Then for $\epsilon >0$ small enough 
 $\det \nabla T_{\epsilon,i} = \det (I + \epsilon \nabla\zeta_i) >0,$ and so $T_{\epsilon,i}: \rt\rightarrow \rt$ is a 
 $C^{\infty}$-diffeomorphism. Let $\rho_{\epsilon,i}$ be the push forward of $\rki$ under the map $T_{\epsilon,i}$ (i.e.,
 $\rho_{\epsilon, i} = T_{\epsilon,i} \# \rki$) and let $\bm{\rho}_{\epsilon} = (\rho_{\epsilon,1}, \ldots, \rho_{\epsilon,n}).$
 Then by change of variable formula \eqref{cv2}
 \begin{align}\label{zxcv0}
  \f(\bm{\rho}_{\epsilon}) = & \ \si \inte \rki(x) \ln \rki(x) \ dx 
  -\si \inte \rki(x) \ln \left(\det \nabla T_{\epsilon,i}(x)\right) \ dx \notag \\
  &+ \si\sj \frac{a_{ij}}{4\pi}\inte\inte \Big(\ln|T_{\epsilon,i}(x)-T_{\epsilon,j}(y)|\Big)\rki(x)\rkj(y) dxdy.
 \end{align}
Since $\nabla\varphi^k_{i} \# \rkai = \rki$ and $T_{\epsilon,i} \# \rki = \rho_{\epsilon,i}$ the map 
$T_{\epsilon,i} \circ \nabla \varphi^k_i$ transports $\rkai$ onto $\rho_{\epsilon,i}$. By definition 
of the Wasserstein distance
\begin{align}\label{zxcv}
\begin{cases}
  d_{\mbox{{\tiny{W}}}}^2(\rki,\rkai) = \inte |x-\nabla \varphi^k_i(x)|^2\rkai(x) \ dx, \\ 
 d_{\mbox{{\tiny{W}}}}^2(\rho_{\epsilon,i},\rkai) 
 \leq \inte |x-T_{\epsilon,i} \circ \nabla\varphi^k_i(x)|^2\rkai(x) \ dx.
\end{cases}
\end{align}
Using the minimizing property of $\rk,$ \eqref{zxcv0} and \eqref{zxcv} we get
\begin{align*}
 0 \leq& \ \f(\bm{\rho}_{\epsilon}) + \frac{1}{2\tau}\bm{d}_{\mbox{{\tiny{W}}}}^2(\bm{\rho}_{\epsilon}, \rka) - \f(\rk) - 
 \frac{1}{2\tau}\bm{d}_{\mbox{{\tiny{W}}}}^2(\rk,\rka) \\
  \leq& \ \frac{1}{2\tau} \left[\si \inte \left(|x-T_{\epsilon,i} \circ \nabla\varphi^k_i(x)|^2
 - |x-\nabla \varphi^k_i(x)|^2\right) \rkai(x) \ dx\right]\\
 &-\si \inte \rki(x) \ln \left(\det \nabla T_{\epsilon,i}(x)\right) \ dx \\
 &+\si\sj \frac{a_{ij}}{4\pi}\inte\inte \left(\ln|T_{\epsilon,i}(x)-T_{\epsilon,i}(y)| - \ln|x-y|\right)\rki(x)\rkj(y) dxdy \\
 =& \ \frac{1}{2\tau} \left[\si \inte \left(|x-\nabla \varphi^k_i(x)-\epsilon \zeta_i \circ \nabla\varphi^k_i(x)|^2 
 - |x-\nabla \varphi^k_i(x)|^2\right) \rkai(x) \ dx\right]\\
 &-\si \inte \rki(x) \ln \left(\det (I+\epsilon \nabla\zeta_i(x))\right) \ dx \\
 &+\si\sj \frac{a_{ij}}{4\pi}\inte\inte \left(\ln|x-y+\epsilon(\zeta_i(x)-\zeta_j(y))|
 - \ln|x-y|\right)\rki(x)\rkj(y) dxdy.
\end{align*}

Dividing by $\epsilon>0$ and letting $\epsilon \rightarrow 0$ we get
\begin{align} \label{el3}
&\frac{1}{\tau}\si\inte (\nabla \varphi^k_i(x)-x) \cdot  \zeta_i(\nabla\varphi^k_i(x))\rkai(x)dx 
    -\si\inte \nabla \cdot \zeta_i(x) \rki(x)dx \notag\\
   &+\si\sj \frac{a_{ij}}{4\pi}\inte\inte \frac{ (\zeta_i(x)- \zeta_j(y)) \cdot (x-y)}{|x-y|^2}
   \rki(x)\rkj(y) \ dxdy \geq 0.
\end{align}
 Finally, fixing an $i$ and choosing $\zeta_i = \pm \zeta \in C_c^{\infty}(\rt;\rt)$ and $\zeta_j= 0$ for all $j 
 \neq i$ we obtain  
\begin{align} \label{el4}
 &\frac{1}{\tau}\inte (\nabla \varphi^k_i(x)-x) \cdot  \zeta(\nabla\varphi^k_i(x))\rkai(x)dx 
    -\inte \nabla \cdot \zeta(x) \rki(x)dx \notag\\
   &+\sj \frac{a_{ij}}{2\pi}\inte\inte \frac{ \zeta(x) \cdot (x-y)}{|x-y|^2}
   \rki(x)\rkj(y) \ dxdy = 0.
\end{align}
Recall that $\nabla u^k_{\tau,i}(x) = -\frac{1}{2\pi}\inte \frac{x-y}{|x-y|^2} \rki(y) \ dy.$
By regularity 
results of Lemma \ref{regularity}, Lemma \ref{regularity u}, $\rki \in L^2(\rt)$ and $|\nabla u^k_{\tau,i}| \in L^2_{loc}(\rt).$
Therefore, the last term in \eqref{el4} makes sense.

Using the definition of $\nabla u^k_{\tau,j}$ in \eqref{el4} and applying integration by parts on the second term 
 we deduce \eqref{el1}.

(b) Let $\varphi^{k,*}_i$ denotes the convex conjugate of $\varphi^k_i.$ Since both $\rki,\rkai \in \g_2,$ in particular,
absolutely continuous with respect to the Lebesgue measure,
$\nabla\varphi^{k,*}_i \# \rki = \rkai.$ 
Using this we can rewrite \eqref{el1} as
\begin{align*}
  \frac{1}{\tau}\inte (x-\nabla \varphi^{k,*}_i(x)) \cdot  \zeta(x)\rki(x)\ dx
   =& -\inte \zeta(x) \cdot \nabla \rki(x) \ dx \\ 
   &+ \sj a_{ij} \inte  \zeta(x) \cdot \nabla u^{k}_{\tau,j}(x)\rki(x) \ dx.
\end{align*}
Since this is true for any $\zeta \in C_c^{\infty}(\rt;\rt)$ we conclude
\begin{align}\label{el5}
 \frac{1}{\tau}\left( x-\nabla \varphi^{k,*}_i(x)\right)\rki(x)
   = - \nabla \rki(x) 
   + \sj a_{ij} \nabla u^{k}_{\tau,j}(x)\rki(x).
\end{align}
Recalling the definition of Wasserstein distance and \eqref{el5} we get
\begin{align*}
 \frac{1}{\tau^2}d_{\mbox{{\tiny{W}}}}^2(\rki, \rkai) &= \frac{1}{\tau^2}\inte  |x-\nabla \varphi^{k,*}_i(x)|^2\rki(x) \ dx \\
 &= \inte \left|\frac{\nabla \rki(x)}{\rki(x)} - \sj a_{ij}\nabla u^k_{\tau,j}(x)\right|^2 \rki(x) dx.
\end{align*}
(c) Finally, for $\psi \in C_c^{\infty}(\rt)$ putting $\zeta = \nabla \psi$ in \eqref{el1}
and using the Taylor expansion formula
\begin{align*}
 \left|\psi(\nabla \varphi_i^k(x)) - \psi(x)
 - (\nabla\varphi_i^k(x) - x) \cdot \nabla \psi(\varphi_i^k(x)) \right| 
 = O\left(||D^2\psi||_{L^{\infty}}\right)
 |x-\nabla \varphi_i^k(x)|^2
\end{align*}
we can rewrite the left hand side of \eqref{el1} as
\begin{align*}
 & \ \frac{1}{\tau}\inte (\nabla \varphi^k_i(x)-x)\cdot  \nabla \psi(\nabla\varphi^k_i(x))\rkai(x)dx \notag \\
 =& \ \frac{1}{\tau} \left[\inte \psi\circ\varphi_i^k(x)\rkai(x) - \inte \psi(x)\rkai(x)\right] 
  + \frac{1}{\tau}O\left(||D^2\psi||_{L^{\infty}}\right)d_{\mbox{{\tiny{W}}}}^2(\rki,\rkai) \\
 =& \ \frac{1}{\tau} \inte \psi(x)\Big(\rki(x) - \rkai(x)\Big) 
  +\frac{1}{\tau}O\left(||D^2\psi||_{L^{\infty}}\right)d_{\mbox{{\tiny{W}}}}^2(\rki,\rkai).
\end{align*}
Inserting the last identity in \eqref{el1} and multiplying by $\tau$ we obtain \eqref{elfinal}.
\end{proof}

\section{Estimates on the time interpolation}
\subsection{Time interpolation} We define the piecewise constant time dependent interpolation
\begin{align*}
 \bm{\rho}_{\tau}(t) = \rk, \ \ if \ t \in ((k-1)\tau, k\tau], \ \ k \geq 1.	
\end{align*}
In the subsequent sections, we will show that for any time $T >0,$ the piecewise constant interpolates $\bm{\rho}_{\tau}$ converges 
{\it in some sense} to a solution (according to definition \eqref{weak sol}) to the PKS-system \eqref{kss}
satisfying the energy dissipation inequality (or, free energy inequality).

\begin{Lem}\label{eti}
 For every $T>0,$ there exists a constant $\mathcal{C}_{in}(T)>0$ such that for every $\tau \in (0,1)$
\begin{align*}
\si \int_0^{T}\inte (\rti(t))^2 \ dxdt 
+ \si\int_{0}^{T} \inte \left|\frac{\nabla \rti(t)}{\rti(t)}\right|^2\rti(t) \ dxdt \leq \mathcal{C}_{in}(T).
\end{align*}
 Moreover, the Newtonian potentials $u_{\tau,i}(t), t\in (0,T)$ satisfy the uniform log decay and locally uniform 
 $H^1$-estimate stated in Lemma \ref{regularity u}. In addition, for every $R>0$ there 
 exists a constant $\mathcal{C}(R,T)$ independent of $\tau$ such that
 \begin{align*}
  \si\int_0^T ||u_{\tau, i}(t)||_{H^2(B_R)}^2 \ dt \leq \mathcal{C}(R,T).
 \end{align*}
 \end{Lem}
\begin{proof}
 Note that by a priori estimate Lemma \ref{apriori}, there exists a constant $\mathcal{C}_{ap}(T)$ such that for all 
 $1 \leq l \leq k$ satisfying $k\tau \leq T$ we have
\begin{align*}
 \si\inte \rli |\ln \rli| \ dx + M_2(\rl) \leq \mathcal{C}_{ap}(T).
\end{align*}

Applying Lemma \ref{regularity}, with $\Theta = \mathcal{C}_{ap}(T)$ we get 
\begin{align*}
 \si \inte \left|\frac{\nabla \rli(x)}{\rli(x)}\right|^2 \rli(x) \ dx \leq \frac{2}{\tau}
 \left[\mathcal{H}(\rla) - \mathcal{H}(\rl)\right] + C(T).
\end{align*}
for some constant $C(T)$ depending on $\mathcal{C}_{ap}(T)$ and for all $l=1,\ldots,k.$
Summing over $l$ and multiplying by $\tau$ gives 
\begin{align*}
 \si \tau\sum_{l=1}^k\inte \left|\frac{\nabla \rli(x)}{\rli(x)}\right|^2 \rli(x) \ dx \leq 
 2\sum_{l=1}^k\left[\mathcal{H}(\rla) - \mathcal{H}(\rl)\right] + C(T)k\tau.
\end{align*}
Using the definition of $\bm{\rho}_{\tau}$ and again using Lemma \ref{apriori} we get
\begin{align} \label{l2est}
 \si \int_{0}^{k\tau}\inte \left|\frac{\nabla \rho_{\tau,i}(t)}{\rho_{\tau,i}(t)}\right|^2 \rho_{\tau,i}(t) \ dxdt \leq 
 2\left[\mathcal{H}(\bm{\rho}^0)-\mathcal{H}(\rk)\right] + C(T)k\tau\leq C_1(T).
\end{align}
The $L^2((0,T)\times \rt)$ estimate follows from Lemma \ref{ce}(a) and \eqref{l2est}.
Finally, by definition  $u_{\tau,i}(t) = u_{\tau,i}^k$ if $t \in ((k-1)\tau, k\tau].$ Therefore, the conclusion
of the second part of the lemma follows from Lemma \ref{regularity u} and proceeding as the proof of 
\eqref{l2est}.
\end{proof}

\begin{Lem} \label{con}
 Let $T>0$ be given. There exists a constant $\mathcal{C}_{con}(T)>0$ such that for all $(s,t)\in [0,T]\times [0,T]$ and 
 $\tau \in (0,1)$ 
 \begin{align*}
 d_{\mbox{{\tiny{W}}}}(\rho_{\tau,i}(t), \rho_{\tau,i}(s)) \leq \mathcal{C}_{con}(T)\left(\sqrt{|t-s|} + \sqrt{\tau}\right) \ \ for \ 
 all \ i=1,\ldots,n. 
 \end{align*}
\end{Lem}
\begin{proof}
 With out loss of generality we can assume that $0 \leq s < t.$ Let $k = [t/\tau]+1$ and $\tilde k = [s/\tau]+1$
 (where $[a]$ denotes the largest integer smaller than $a$). Then using the definition of $\bm{\rho}_{\tau}$
 and Lemma \ref{apriori} we get
 \begin{align*}
 d_{\mbox{{\tiny{W}}}}(\rho_{\tau,i}(t), \rho_{\tau,i}(s))=
&d_{\mbox{{\tiny{W}}}}(\rho_{\tau,i}^{k}, \rho_{\tau,i}^{\tilde k})
\leq \sum_{l=\tilde k}^{k-1} d_{\mbox{{\tiny{W}}}}(\rho_{\tau,i}^{l+1}, \rho_{\tau,i}^{l}) 
\leq \sqrt{k - \tilde k}\left[\sum_{l=\tilde k}^{k} d_{\mbox{{\tiny{W}}}}^2(\rho_{\tau,i}^{l+1}, \rho_{\tau,i}^{l})\right]^{\frac{1}{2}}\\
&\leq \sqrt{k - \tilde k}\sqrt{\mathcal{C}_{ap}(T)\tau}
\leq \mathcal{C}_{con}(T)\left(\sqrt{|t-s|} + \sqrt{\tau}\right). 
\end{align*}
\end{proof}

\section{Convergence of the MM scheme}
We are going to apply the refined Arzel\`{a}-Ascoli Theorem \ref{ArzelaAscoli} (see appendix) to 
the following choices:
$\mathcal{S} = \mathcal{P}_2^{\beta_i}(\rt),$
$\sigma$ = narrow convergence in $\mathcal{P}_2^{\beta_i}(\rt),$
$K = \{\rho_{\tau,i}(t)| \ t \in [0,T], \ \tau \in (0,1)\},$
$d = d_{\mbox{{\tiny{W}}}},$
$w(s,t) = \sqrt{|s-t|},$ note that here $\mathcal{B} = \emptyset.$

\subsection{Convergence of the time interpolation} By Lemma \ref{apriori}
\begin{align*}
 \sup_{t \in [0,T], \tau \in (0,1)} \si\left(\inte \rho_{\tau,i}(t)|\ln \rho_{\tau,i}(t)| + M_2(\rho_{\tau,i}(t))\right) < +\infty.
\end{align*}
Therefore the set $\{\rho_{\tau,i}(t)| \ t \in [0,T], \ \tau \in (0,1)\}$ is weakly 
sequentially compact in $ L^1(\rt)$ (in particular, sequentially compact with respect to $\sigma$).
Moreover, by Lemma \ref{con}
\begin{align*}
 \limsup_{\tau \rightarrow 0} d_{\mbox{{\tiny{W}}}}(\rho_{\tau,i}(t), \rho_{\tau,i}(s)) \leq 
 \mathcal{C}_{con}(T)\sqrt{|t-s|}.
\end{align*}
Therefore by Theorem \ref{ArzelaAscoli}, there exists a curve $\bm{\rho}(t): [0,T]\rightarrow (L^1(\rt))^n$ and 
a monotone decreasing sequence $\tau_m \rightarrow 0$ such that 
\begin{itemize}
 \item $\rho_{\tau_m, i}(t) \rightharpoonup \rho_i(t)$ weakly in $L^1(\rt)$ for every $t \in [0,T].$
 \item $\rho_i$ is $d_{\mbox{{\tiny{W}}}}$ continuous. More precisely, $d_{\mbox{{\tiny{W}}}}(\rho_i(t),\rho_i(s)) 
 \leq \mathcal{C}_{con}(T)\sqrt{|t-s|}.$ Thus
 $\rho_i \in C^{0,\frac{1}{2}}([0,T];\mathcal{P}_2^{\beta_i}(\rt))$ for $i=1,\ldots,n.$
\end{itemize}
Furthermore, by Lemma \ref{eti}, $\rho_{\tau_m,i}$ is bounded in $L^2((0,T)\times \rt)$
and hence
\begin{align*}
 \rho_{\tau_m, i} \rightharpoonup \rho_i \ \ weakly \ in \ L^2((0,T)\times \rt),
\end{align*}
proving that $\rho_i \in L^2((0,T)\times \rt).$ 
Finally, consider the sequences
\begin{align*}
 d\mu_m = \frac{1}{T\beta_i}\rho_{\tau_m,i}dxdt, \ \ \tilde v_m = \left(\frac{\nabla_x \rho_{\tau_m,i}}{\rho_{\tau_m,i}}, 1\right)
 =(v_m,1).
\end{align*}
By Lemma \ref{eti} and Proposition \ref{cvf}, applied to the probability measures $d\mu_m$ 
and the vector fields $\tilde v_m$ we see that there exists a vector
field $v \in L^2((0,T)\times\rt, \rho_i\ ;\rt)$ such that 
\begin{align*}
 \int_0^T\inte \zeta \cdot v_m\rho_{\tau_m,i} \ dxdt\rightarrow \int_0^T\inte \zeta \cdot v \rho_i \ dxdt, \ \ 
\ \ for \ all \ \zeta \in C_c^{\infty}((0,T)\times \rt ; \rt).
 \end{align*}
Moreover, since $v\rho_i \in L^1((0,T)\times \rt;\rt)$ and
\begin{align}\label{nabcon}
 \int_0^T\inte \zeta \cdot v_m\rho_{\tau_m,i} &= \int_0^T\inte \zeta \cdot \nabla_x \rho_{\tau_m,i} 
 &=-\int_0^T\inte (\nabla_x \cdot \zeta)  \rho_{\tau_m,i}\rightarrow
-\int_0^T\inte (\nabla_x \cdot \zeta)  \rho_i,
\end{align}
we conclude $v\rho_i = \nabla_x \rho_i$ and hence $\rho_i \in L^1((0,T);W^{1,1}(\rt)).$ By lower semicontinuity
(Proposition \ref{cvf} equation \eqref{lscv}) and Lemma \ref{eti}
\begin{align}\label{7.3}
 \int_0^T \inte \left|\frac{\nabla \rho_i(t)}{\rho_i(t)}\right|^2\rho_i(t) \ dxdt < \infty,
\end{align}
proving that $\frac{\nabla \rho_i}{\rho_i} \in L^2((0,T) \times \rt, \rho_i \ ; \rt)$ for all $i \in \{1,\ldots,n\}.$
  
\subsection{Convergence of the Newtonian potential}
This is a crucial step because apparently we do not have explicit continuity in time estimates on $\bm{u}_{\tau_m}$
as in Lemma \ref{con} for $\bm{\rho}_{\tau_m}$. Fortunately, we do have continuity in time 
estimate on the harmonic part of $\bm{u}_{\tau_m},$ which together with the results of section $7.1$
enables us to prove the following lemma: set
 \begin{align*}
  \omega(x) = \frac{1}{(1+|x|^2)^2}
 \end{align*}

 \begin{Lem}\label{conv}
Up to a subsequence  $u_{\tau_m,i}(t) \rightarrow u_i(t)$ strongly in $ L^2(0,T;L^2(\rt,\omega))$. 
 \end{Lem}

\begin{proof}
Since $\rho_i \in L^2((0,T)\times\rt)),$ for almost every $t \in [0,T], \ ||\rho_i(t)||_{L^2(\rt)} < +\infty.$  
Denote $\iota := \{t \in [0,T]| \ ||\rho_i(t)||_{L^2(\rt)} = +\infty\}.$
By Lemma \ref{eti}, for each $t \in [0,T], \ u_{\tau_m,i}(t)$ is bounded in $H^1_{loc}(\rt).$
Hence we can extract a countable dense subset $\mathcal{K}$ of $[0,T]\backslash \iota$ and a 
subsequence (indexed by $\tau_m$ itself) such that 
\begin{align*}
 u_{\tau_m, i}(t) \rightarrow u_i(t) \ strongly \ in \ L^2_{loc}(\rt) \ \ for \  all \ t  \in \mathcal{K}, \ i=1,\ldots,n. 
\end{align*}
Since, by Lemma \ref{eti}, $u_{\tau_m,i}$ has uniform log decay  we deduce
\begin{align*}
  u_{\tau_m, i}(t) \rightarrow u_i(t) \ strongly \ in \ L^2(\rt,\omega) \ \ for \  all \ t \in  \mathcal{K}, \ i=1,\ldots,n. 
\end{align*}

Next we need some continuity in time estimates on $u_{\tau_m,i}$ in order to conclude the above 
$L^2$-convergence holds for all $t \in [0,T]\backslash \iota.$

\vspace{0.2 cm}

{\bf Continuity estimate 1:} by Lemma \ref{con}, for every $s,t \in [0,T]$ 
and $\psi \in C_c^{\infty}(\rt)$  
\begin{align}\label{conest1}
 \left|\inte (u_{\tau_m,i}(t) - u_{\tau_m,i}(s))\Delta\psi \ dx\right| &= 
 \left|\inte (\rho_{\tau_m,i}(t) - \rho_{\tau_m,i}(s))\psi \ dx\right| \notag\\
 &\leq ||\nabla\psi||_{L^{\infty}(\rt)}d_{\mbox{{\tiny{W}}}}(\rho_{\tau_m,i}(t), \rho_{\tau_m,i}(s)) \notag\\
 &\leq \mathcal{C}_{con}(T)||\nabla\psi||_{L^{\infty}(\rt)}\left(\sqrt{|t-s|} + \sqrt{\tau_m}\right).
\end{align}

{\bf Continuity estimate 2:} the quadratic interaction term is continuous 
with respect to the weak $L^1$-convergence. Particularly, for each $t \in [0,T],$
\begin{align} \label{conest2}
\lim_{m\rightarrow \infty} \inte u_{\tau_m,i}(t)\rho_{\tau_m,i}(t) \ dx 
 =\lim_{m\rightarrow \infty}\inte u_{\tau_m,i}(t)\rho_{i}(t) \ dx
 = \inte u_i(t) \rho_i(t) \ dx.
\end{align}
For a proof of \eqref{conest2} see \cite[Lemma $2.3$ and the proof of Lemma $3.1$]{JKO2}. 
Here the crucial point is that $\rho_{\tau_m,i}(t)\rightharpoonup \rho_i(t)$ weakly in $L^1(\rt)$ for 
all $t \in [0,T]$ and has uniform entropy and second moment bound.

Now take any $s \in [0,T] \backslash (\mathcal{K}\cup \iota).$ Let $u_{\tau_m^{\prime},i}(s)$
be a subsequence of $u_{\tau_m,i}(s)$ such that $u_{\tau_m^{\prime},i}(s) \rightarrow v_i(s)$ 
strongly in $L^2_{loc}(\rt)$
for some $v_i(s) \in H^1_{loc}(\rt).$ Along the same subsequence passing to the limit in \eqref{conest1} we get
\begin{align}\label{b}
 \left|\inte (u_i(t) - v_i(s))\Delta\psi \ dx\right| \leq C ||\nabla \psi||_{L^{\infty}}\sqrt{|t-s|}, \ for \ t \in \mathcal{K},
 \ s \in [0,T]\backslash (\mathcal{K}\cup \iota).
\end{align}
This, and $\rho_i \in C^{0,\frac{1}{2}}([0,T];\mathcal{P}_2^{\beta_i}(\rt))$ imply that 
\begin{align*}
 \inte (u_i(s) - v_i(s))\Delta \psi \ dx = 0, \ \ for  \ all \ \psi \in C_c^{\infty}(\rt).
\end{align*}

Hence $u_i(s) - v_i(s)$ is Harmonic in $\rt.$ But both have log decay and therefore $u_i(s)-v_i(s)$
must be a constant. Denote this constant by $\varepsilon_i.$ We claim that $\varepsilon_i = 0.$
Indeed, by strong $L^2$-convergence of $u_{\tau_m^{\prime},i}(s)$ and using $\rho_i(s) \in L^2(\rt)$
we obtain
\begin{align}\label{b1}
 \inte u_{\tau_m^{\prime},i}(s) \rho_i(s) \ dx \rightarrow \inte v_i(s)\rho_i(s) \ dx 
 = \inte u_i(s)\rho_i(s) \ dx - \epsilon_i\beta_i.
\end{align}
On the other hand by \eqref{conest2}
\begin{align}\label{b2}
 \inte u_{\tau_m^{\prime},i}(s) \rho_i(s) \rightarrow   \inte u_i(s)\rho_i(s).
\end{align}
Hence \eqref{b1} and \eqref{b2} we get $\epsilon_i = 0.$
By Lemma \ref{eti}, $||u_{\tau_m, i}(t)||_{L^2(\rt,\omega)}^2 \leq C(T)$ for all $t \in [0,T].$
We can apply dominated convergence theorem to conclude the Lemma. 
\end{proof}

Let $\xi \in C_c^{\infty}([0,T) \times \rt)$ and $\psi \in C_c^{\infty}(\rt)$ be a smooth compactly supported function such that $\psi = 1$ on a domain 
containing the support of 
$\xi(t)$ for all $t.$ Then again by Lemma \ref{eti}, the sequence $\{\psi u_{\tau_m, i}\}$ 
is bounded in $L^2(0,T; H^2(\rt)).$
By Lemma \ref{conv}, Lemma \ref{compactness in l2}  (applied to $X = H^2(\rt), B = H^1(\rt,\omega),
Y = L^2(\rt,\omega), 
F = \{\psi u_{\tau_m,i}\}$)
we deduce
\begin{align}\label{conv u}
   \nabla \xi \cdot \nabla u_{\tau_m,j} \rightarrow  \nabla \xi \cdot \nabla u_{j}
 \ \ in \ L^2(0,T;L^2(\rt)) \ \ for \ all \ j=1,\ldots,n.
\end{align}

\subsection{Rewriting the Euler-Lagrange equation}
Set $N_m := [T/\tau_m]$ and take $\xi$ of class $ C_c^{\infty}([0,T)\times \rt).$ Then 

\begin{align}\label{cov1}
 &-\int_0^{N_m\tau_m} \inte \partial_t\xi \rho_{\tau_m,i}(t) = \notag \\
\sum_{k=1}^{N_m}\inte \xi(\cdot, (k-1)\tau_m)&\left(\rho_{\tau_m,i}^k - \rho_{\tau_m,i}^{k-1}\right)
 - \inte \xi(\cdot, N_m\tau_m)\rho_{\tau_m,i}^{N_m} 
 + \inte \xi(0) \rho_i^0(x)
\end{align}
Since $\xi \equiv 0$ at $t=T$
\begin{align}\label{cov2}
\left|\inte \xi(N_m\tau_m)\rho_{\tau_m,i}^{N_m}\right| 
 \leq \inte |\xi(N_m\tau_m)-\xi(T)|\rho_{\tau_m,i}^{N_m} = O(\tau_m), 
\end{align}
where we used $|T-N_m\tau_m|\leq \tau_m.$ 
For each $k \in \{1,\ldots,N_m\}$ applying \eqref{elfinal}(Lemma \ref{el}(c)) with 
$\psi =  \xi((k-1)\tau_m)$ we get
\begin{align*}
 &\left|\inte \xi(\cdot, (k-1)\tau_m)\left(\rkim(x) - \rkaim(x)\right) + \tau_m\inte  \nabla\xi(\cdot,(k-1)\tau_m) \cdot \nabla \rkim \right. \notag\\
 &\left. -\tau_m\sj a_{ij} \inte  \nabla \xi(\cdot,(k-1)\tau_m) \cdot \nabla u^{k}_{\tau_m,j}(x)\rkim(x)dx \right| 
 = O(||D^2\xi||_{L^{\infty}}) d_{\mbox{{\tiny{W}}}}^2(\rkim,\rkaim).
\end{align*}
Summing over $k = 1,\ldots, N_m$ and using Lemma \ref{apriori} we get
\begin{align}\label{cov3}
 &\sum_{k=1}^{N_m}\inte \xi(\cdot,(k-1)\tau_m)\left(\rkim(x) - \rkaim(x)\right) + \sum_{k=1}^{N_m}\tau_m\inte \nabla\xi(\cdot,(k-1)\tau_m) \cdot \nabla \rkim  \notag\\
 & -\sum_{k=1}^{N_m}\tau_m\sj a_{ij} \inte \nabla \xi(\cdot,(k-1)\tau_m) \cdot \nabla u^{k}_{\tau_m,j}(x)\rkim(x)dx = O(\tau_m).
\end{align}
We write $(I_1) + (I_2) + (I_3)$ for the terms on the left hand side of \eqref{cov3}.
Thanks to \eqref{cov1} and \eqref{cov2} we can write
\begin{align*}
 (I_1) = -\int_0^{T} \inte \partial_t\xi \rho_{\tau_m,i} 
 - \inte \xi(0) \rho_i^0(x) + O(\tau_m).
\end{align*}
For $(I_2)$ and $(I_3)$ we use the following estimates:
\begin{align*}
& \left|\int_{(k-1)\tau_m}^{k\tau_m} \inte  \nabla\xi \cdot \nabla \rkim -\tau_m\inte \nabla\xi(\cdot,(k-1)\tau_m) 
\cdot \nabla \rkim \right|\\
 &\leq \int_{(k-1)\tau_m}^{k\tau_m} \inte |\nabla\xi - \nabla\xi(\cdot, (k-1)\tau_m)||\nabla \rkim| \\
& \leq ||D^2_{xt}\xi||_{L^{\infty}}||\nabla \rkim||_{L^1(\rt)}\tau_m^2 =O(\tau_m^{\frac{3}{2}}),
\end{align*}
where we have used $||\nabla \rkim||_{L^1(\rt)}\leq \sqrt{\beta_i}||\nabla \rkim/\rkim||_{L^2(\rt,\rkim)} 
= O(1/\sqrt{\tau_m}),$ thanks to Lemma \ref{regularity}.
 Similarly,
\begin{align*}
 &\tau_m \inte  \nabla \xi(\cdot, (k-1)\tau_m)\cdot \nabla u^{k}_{\tau_m,j}(x)\rkim(x)dx \\
&= \int_{(k-1)\tau_m}^{k\tau_m} \inte   \nabla \xi \cdot \nabla u^{k}_{\tau_m,j}(x)\rkim(x)dx 
+ O(\tau_m^{\frac{3}{2}}).
 \end{align*}
 By Cauchy-Schwartz and Lemma \ref{regularity u} 
\begin{align}\label{3ineq}
 \begin{cases}
  -\int_{N_m\tau_m}^T \inte \partial_t\xi(t) \rho_{\tau_m,i}(t) = O(\tau_m), \\
\int_{N_m\tau_m}^T \inte \nabla\xi(t) \cdot \nabla \rho_{\tau_m,i}(t)  = O(\tau_m^{\frac{1}{2}}), \\
 \sj a_{ij}\int_{N_m\tau_m}^T \inte   \nabla \xi(t) \cdot \nabla u_{\tau_m,j}(t)\rho_{\tau_m,i}(t)dx 
 = O(\tau_m^{\frac{1}{2}}).
 \end{cases}
\end{align}

Therefore from \eqref{cov1}-\eqref{cov3} combined with the three estimates in \eqref{3ineq} we get
\begin{align}\label{final el}
 -\int_0^{T} \inte \partial_t\xi\rho_{\tau_m,i} &- \inte \xi(0)\rho_i^0
 +\int_{0}^{T} \inte \nabla\xi \cdot \nabla \rho_{\tau_m,i} \notag \\
 &-\sj a_{ij}\int_{0}^{T} \inte  \nabla \xi \cdot \nabla u_{\tau_m,j}\rho_{\tau_m,i}dx 
= O(\tau_m^{\frac{1}{2}}).
\end{align}

\subsection{Passing to the limit} 
  For the first term in \eqref{final el} we use $\rho_{\tau_m,i} \rightharpoonup \rho_i$ in $L^2((0,T)\times \rt),$
  for the third term we use \eqref{nabcon}, 
  and finally for the last term we use \eqref{conv u}, the weak convergence of $\rho_{\tau_m,i}$ to $\rho_i$ in $L^2((0,T)\times \rt)$ and
  duality to pass to the limit. Therefore we conclude $\rho_i$ satisfies 
\begin{align*}
 -\int_0^{T} \inte \partial_t\xi(t) &\rho_i - \inte \xi(0)\rho_i^0
 +\int_{0}^{T} \inte  \nabla\xi \cdot \nabla \rho_i 
 -\sj a_{ij}\int_{0}^{T} \inte  \nabla \xi(t) \cdot \nabla u_j\rho_idx = 0.
\end{align*}
for all $i=1,\ldots,n$ and any test function $\xi \in C_c^{\infty}([0,T)\times \rt).$ Which is the weak formulation of the 
PKS-system \eqref{kss}
(see Definition \ref{weak sol}).
 
\section{Free energy inequality}
In this section, we show, by using De Giorgi variational interpolation, that the obtained solution 
satisfies the free energy inequality. 
Define for $\tau \in (0,1)$
\begin{align}\label{degiorgi}
\bm{\tilde \rho}_{\tau}(t) := \arg \min_{\bm{\rho}\in \g_2} \left\{\f(\bm{\rho}) + \frac{1}{2(t - (k-1)\tau)}
\bm{d}_{\mbox{{\tiny{W}}}}^2(\bm{\rho},\bm{\rho}_{\tau}^{k-1})\right\}, \ \ t \in ((k-1)\tau, k\tau]. 
\end{align}
With out loss of generality we can assume that $\bm{\tilde \rho}_{\tau}(k\tau) = \bm{\rho}_{\tau}^k.$

\begin{Rem} \label{remdeg}
One can proceed as in Lemma \ref{apriori} and obtain local uniform entropy and 
moment bound on $\bm{\tilde \rho}_{\tau}(t).$ As a consequence, for each $t, \ \tilde \rho_{\tau,i}(t)$ 
exhibits the same regularity properties stated in the first part of the Lemma \ref{regularity}. 
It is not, however, clear that $\bm{\tilde \rho}_{\tau}(t)$ has finite Fisher information 
bound, which is essential to obtain the $H^2$-regularity estimate
on the Newtonian potential of $\bm{\tilde \rho}_{\tau}(t)$ stated in 
Lemma \ref{eti}. 

The following discrete energy identity serves the purpose in this
regard and it is also an imperative step to prove the free energy inequality.
\end{Rem}

\begin{Lem}[Discrete energy identity] \label{dei}
 For every $k \in \mathbb{N}$ and $\tau \in (0,1)$ the De-Giorgi interpolation defined by \eqref{degiorgi} 
 satisfies the following energy identity:
\begin{align*}
&\si\frac{1}{2} \int_{0}^{k\tau}  \inte \left|\frac{\nabla \rho_{\tau,i}}{\rho_{\tau,i}} 
 - \sj a_{ij}\nabla u_{\tau,j}\right|^2 \rho_{\tau,i} \ dxdt \\
 &+ \si\frac{1}{2} \int_0^{k\tau}
 \inte \left|\frac{\nabla \tilde \rho_{\tau,i}}{\tilde \rho_{\tau,i}} - \sj a_{ij}\nabla \tilde u_{\tau,j}\right|^2 
 \tilde \rho_{\tau,i} \ dxdt + \f(\bm{\rho}_{\tau}(k\tau)) = \f(\bm{\rho}^0).
\end{align*}
Furthermore, for every $T>0$ there exists a constant $C(T)>0$ such that 
\begin{align}\label{same limit}
 \bm{d}_{\mbox{{\tiny{W}}}}(\bm{\rho}_{\tau}(t), \bm{\tilde \rho}_{\tau}(t)) \leq C(T)\tau, \ \ for \ all \ t \in [0,T].
\end{align}
 \end{Lem}

\begin{proof}
 Proceeding exactly as in \cite[Theorem $3.1.4$ and Lemma $3.2.2$]{AGS} we have
 \begin{align*}
  \frac{1}{2}\frac{\bm{d}_{\mbox{{\tiny{W}}}}^2(\rl,\rla)}{\tau} + \frac{1}{2}\int_{(l-1)\tau}^{l\tau}\frac{\bm{d}_{\mbox{{\tiny{W}}}}^2
  (\bm{\tilde \rho}_{\tau}(t),  \rla)}{(t-(l-1)\tau)^2}dt + \f(\rl) = \f(\rla), \ \ l \geq 1.
 \end{align*}
Summing over all $l \in \{1,\ldots,k\}$ we deduce	
\begin{align}\label{fe1}
  \frac{1}{2}\sum_{l=1}^k\frac{\bm{d}_{\mbox{{\tiny{W}}}}^2(\rl,\rla)}{\tau}
  + \frac{1}{2}\int_{0}^{k\tau} G_{\tau}(t)^2dt + \f(\rk) = \f(\bm{\rho}^0),
\end{align}
where 
\begin{align*}
 G_{\tau}(t) = \frac{\bm{d}_{\mbox{{\tiny{W}}}}
  (\bm{\tilde \rho}_{\tau}(t),  \rla)}{t-(l-1)\tau}, \ \ t \in ((l-1)\tau, l\tau].
\end{align*}
Recall that by Lemma \ref{el}(b)
\begin{align}\label{fe2}
 \frac{1}{\tau^2}d_{\mbox{{\tiny{W}}}}^2(\rli, \rlai) = 
 \inte \left|\frac{\nabla \rli}{\rli} - \sj a_{ij}\nabla u^l_{\tau,j}\right|^2 \rli dx.
\end{align}
Applying the same argument of Lemma \ref{el} to $\bm{\tilde \rho}_{\tau}$ we infer that
\begin{align}\label{fe3}
 \frac{d_{\mbox{{\tiny{W}}}}^2(\tilde \rho_{\tau,i}(t), \rlai)}{(t-(l-1)\tau)^2} = 
 \inte \left|\frac{\nabla \tilde \rho_{\tau,i}}{\tilde \rho_{\tau,i}} - \sj a_{ij}\nabla \tilde u_{\tau,j}\right|^2 
 \tilde \rho_{\tau,i} dx, 
\end{align}
for all  $t \in ((l-1)\tau, l\tau].$
Plugging \eqref{fe2} and \eqref{fe3} into \eqref{fe1} and using the definition of $\bm{\rho}_{\tau}$ we obtain
the aforementioned discrete energy identity. 
The proof of \eqref{same limit} is similar to the proof of \cite[Lemma $3.2.2$ equation $(3.2.7)$]{AGS}.
\end{proof}

\begin{Lem}\label{dgie}
For every $T \in (0,\infty)$ there exists a constant $\mathcal{C}_{de}(T) >0$ such that for every 
$\tau \in (0,1),$ the 
De Giorgi interpolates defined by \eqref{degiorgi} satisfies
\begin{align*}
 \int_0^T \inte \left|\frac{\nabla \tilde\rho_{\tau, i}}{\tilde\rho_{\tau, i}}\right|^2 
 \tilde\rho_{\tau, i}\ dx dt \leq \mathcal{C}_{de}(T).
\end{align*}
\end{Lem}

\begin{proof}
By Lemma \ref{dei} and Lemma \ref{apriori} we get 
\begin{align}\label{fi}
\si \int_0^{T}
 \inte \left|\frac{\nabla  \tilde\rho_{\tau,i}}{\tilde \rho_{\tau,i}} - \sj a_{ij}\nabla \tilde u_{\tau,j}\right|^2 
 \tilde \rho_{\tau,i} \ dxdt \leq C(T)
\end{align}
for some constant $C(T)>0$ which implicitly depends on $\mathcal{C}_{ap}(T).$
Now can write for an $t\in [0,T]$
\begin{align}\label{z1}
 &\si \inte \left|\frac{\nabla  \tilde\rho_{\tau,i}}{\tilde \rho_{\tau,i}}
 - \sj a_{ij}\nabla \tilde u_{\tau,j}\right|^2 
 \tilde \rho_{\tau,i} \ dx \notag\\
 \geq &\si \inte \frac{|\nabla\tilde\rho_{\tau,i}|^2}{\tilde\rho_{\tau,i}} \ dx
 - 2 \si\sj a_{ij} \inte  \nabla\tilde\rho_{\tau,i} \cdot \nabla \tilde u_{\tau,j} \ dx \notag\\
 \geq &\si \inte \frac{|\nabla\tilde\rho_{\tau,i}|^2}{\tilde\rho_{\tau,i}} \ dx
 - 2 \si\sj a_{ij} \inte \tilde \rho_{\tau,i}\tilde \rho_{\tau,j}\ dx \notag \\
\geq &\si \inte \frac{|\nabla\tilde\rho_{\tau,i}|^2}{\tilde\rho_{\tau,i}} \ dx
- 4n(\max_{i,j \in I} a_{ij}) \si \inte \tilde\rho_{\tau,i}^2\ dx. 
 \end{align}

 Next we make use of the following inequality whose proof can be found in \cite[Lemma $3.2$]{NMS}(see also \cite[Lemma $2.1$]{FM}):
 There exists a constant $C_p>0$ such that 
 \begin{align}\label{z}
  ||f||_{L^p(\rt)} \leq C_p ||f||_{L^1(\rt)}^{\frac{1}{p}}\left(\inte \frac{|\nabla f|^2}{f} \ dx\right)^{1-\frac{1}{p}}
 \end{align}
holds true for all $p \in [1,\infty)$ and for all $f\in L^1_+(\rt)$ satisfying $\inte \frac{|\nabla f|^2}{f} \ dx <\infty.$
 
 Choose $M>1$ a large number whose value will be decided later. An application of H\"{o}lder inequality and 
 \eqref{z} with $p=3$  gives, as in \cite{NMS}
 \begin{align} \label{z4}
  \inte \tilde\rho_{\tau,i}^2&\mathbbm{1}_{\{\tilde\rho_{\tau,i} >M\}} \ dx 
  \leq \left(\inte \tilde\rho_{\tau,i}
  \mathbbm{1}_{\{\tilde\rho_{\tau,i} >M\}}\ dx\right)^{\frac{1}{2}} \left(\inte\tilde\rho_{\tau,i}^3\ dx
  \right)^{\frac{1}{2}},\notag\\
  &\leq \left(\frac{1}{ln M}\inte \tilde\rho_{\tau,i}
  |\ln\tilde\rho_{\tau,i}| \ dx\right)^{\frac{1}{2}} \left(C_3^{\frac{3}{2}}\beta_i^{\frac{1}{2}}
  \inte\frac{|\nabla\tilde\rho_{\tau,i}|^2}{\tilde\rho_{\tau,i}} \ dx
  \right),
 \end{align}
where $\mathbbm{1}_A$ denotes the characteristic function of a set $A.$ Choosing $M$ large and utilizing \eqref{z4}
and the entropy bound (see Remark \ref{remdeg}) we can estimate 
\begin{align}\label{z2}
 4n (\max_{i,j \in I} a_{ij})\si\inte \tilde\rho_{\tau,i}^2\mathbbm{1}_{\{\tilde\rho_{\tau,i} >M\}} \ dx
 \leq \frac{1}{2} \si \inte\frac{|\nabla\tilde\rho_{\tau,i}|^2}{\tilde\rho_{\tau,i}} \ dx.
\end{align}
On the other hand 
\begin{align}\label{z3}
 4n (\max_{i,j \in I} a_{ij})\si\inte \tilde\rho_{\tau,i}^2\mathbbm{1}_{\{\tilde\rho_{\tau,i} \leq M\}} \ dx
 \leq 4n (\max_{i,j \in I} a_{ij})M\si \beta_i.
\end{align}
Combining \eqref{z1},\eqref{z2} and \eqref{z3} we deduce for every $t \in [0,T]$
\begin{align*}
 \frac{1}{2} \si \inte\frac{|\nabla\tilde\rho_{\tau,i}(t)|^2}{\tilde\rho_{\tau,i}(t)} \ dx
 \leq & \si \inte \left|\frac{\nabla  \tilde\rho_{\tau,i}(t)}{\tilde \rho_{\tau,i}(t)} - \sj a_{ij}\nabla \tilde u_{\tau,j}(t)\right|^2 
 \tilde \rho_{\tau,i}(t) \ dx \\
 &+4n (\max_{i,j \in I} a_{ij})M\si \beta_i.
\end{align*}
Integrating with respect to $t$ from $0$ to $T$ and using \eqref{fi} we conclude the proof.
\end{proof}

Finally, we are in a position to prove the free energy inequality.

\vspace{0.2 cm}

\noindent
{\bf Proof of Theorem \ref{main}$(b)$:}
\begin{proof} 
Fix $T>0$ and choose the sequence $\tau_m$ as in 
section $7$ so that $\rho_{\tau_m,i}(t) \rightharpoonup \rho_{i}(t)$ weakly in $L^1(\rt)$ for all $t \in [0,T]$
and $\rho_{\tau_m,i} \rightharpoonup \rho_i$ weakly in $L^2((0,T)\times \rt).$ 
Moreover, the Newtonian potentials $u_{\tau_m,i} \rightarrow u_{i}$ strongly in $L^2((0,T);H^1_{loc}(\rt)).$

Furthermore by Lemma \ref{dei}, Lemma \ref{dgie} and Remark \ref{remdeg}, 
the De-Giorgi interpolation $\bm{\tilde \rho}_{\tau_m}$ 
enjoys the same property and also converges to the same limit $\bm{\rho}.$ Set 
\begin{align*}
 v_{m,i} = \frac{\nabla \rho_{\tau_m,i}}{\rho_{\tau_m,i}} 
 - \sj a_{ij}\nabla u_{\tau_m,j}, 
\end{align*}
$\tilde v_{m,i} = (v_{m,i},1)$ and $d\mu_{m,i} = \frac{1}{T\beta_i}\rho_{\tau_m,i} dxdt.$ Note that by Lemma \ref{dei}
and Lemma \ref{apriori}
\begin{align*}
 \sup_m ||\tilde v_{m,i}||_{L^2((0,T)\times \rt, \mu_{m,i}; \ \mathbb{R}^3)} < +\infty.
\end{align*}
Invoking Proposition \ref{cvf} again we obtain the existence a vector 
field $v_i \in L^2((0,T)\times \rt, \rho_i; \mathbb{R}^2)$ such that 
\begin{align*}
 \int_0^T\inte \zeta \cdot v_{m,i}\rho_{\tau_m,i} \ dxdt\rightarrow \int_0^T\inte \zeta \cdot v_i \rho_i \ dxdt, \ \ \  \ for \ all \ \zeta 
 \in C_c^{\infty}((0,T)\times \rt; \ \rt).
\end{align*}
Now proceeding as the proof of \eqref{nabcon} goes, we derive
\begin{align*}
 \int_0^T\inte \zeta \cdot v_i \rho_i = \lim_{m \rightarrow +\infty}\int_0^T\inte \zeta \cdot v_{m,i}\rho_{\tau_m,i} 
= \int_0^T\inte \zeta \cdot \left(\nabla \rho_{i}
 - \sj a_{ij}\nabla u_{j} \rho_{i}\right).
 \end{align*}
Since $v_i\rho_i \in L^1((0,T)\times \rt)$ we conclude $v_i\rho_i = \nabla \rho_{i}
 - \sj a_{ij}\nabla u_{j} \rho_{i}$ and moreover, by lower semicontinuity (Proposition \ref{cvf} \eqref{lscv})
 \begin{align} \label{ineq}
 &\int_{0}^{T}  \inte \left|\frac{\nabla \rho_{i}(t)}{\rho_{i}(t)} 
 - \sj a_{ij}\nabla u_{j}(t)\right|^2 \rho_{i}(t) \ dxdt \notag\\
 &\leq \liminf_{m\rightarrow +\infty}
 \int_{0}^{T}  \inte \left|\frac{\nabla \rho_{\tau_m,i}(t)}{\rho_{\tau_m,i}(t)} 
 - \sj a_{ij}\nabla u_{\tau_m,j}(t)\right|^2 \rho_{\tau_m,i}(t) \ dxdt
\end{align}
holds. Similarly, \eqref{ineq} holds for $\bm{\tilde \rho}_{\tau_m}.$ 
 Finally, passing to the limit in the discrete energy identity (Lemma \ref{dei})
and using the lower semi-continuity of $\f$ with respect to the narrow convergence we get 
\begin{align*}
 \si\int_{0}^{T}  \inte \left|\frac{\nabla \rho_{i}(t)}{\rho_{i}(t)} 
 - \sj a_{ij}\nabla u_{j}(t)\right|^2 \rho_{i}(t) \ dxdt + \f(\bm{\rho}(T)) \leq \f(\bm{\rho}^0).
\end{align*}
This completes the proof of Theorem \ref{main}(b).
\end{proof}

\pagebreak
\noindent
{\bf Proof of Theorem \ref{main}$(c)$:}

\begin{proof} 
The proof of uniqueness follows directly as in  the scalar case ($n=1$)  \cite{FM}, so we only sketch the main steps.  First,  prove the hypercontractivity result \cite{Gross, BD}  for the weak solution,  
using  a variant of Diperna-Lions renormalizing trick in spirit of \cite{NMS} and \cite{FM}, 
 extended for the system. For this we apply   the a-posteriori estimate  (\ref{7.3}) from part (b) of the Theorem. As a result we get that any weak solution with finite initial entropy is smooth at $t>0$. Then we get the estimate 
\be\label{Pdef}
\lim_{t\rightarrow 0+}P(t)=0 \ \text{where} \ 
 P(t):=  t^{\frac{1}{4}} \si ||\rho_i(t)||_{L^{\frac{4}{3}}(\rt)} .
\ee

 Next  we consider two such weak solutions $\rho_i^0, \rho_i^1$ which agrees at $t=0$, and let $ Q(t) = \sup_{0<s \leq t} \si s^{\frac{1}{4}}||F_i(s)||_{L^{\frac{4}{3}}(\rt)}$ where $F_i(s):= \rho^0_i(\cdot, s)-\rho^1_i(\cdot, s)$.
 Then we get the estimate for $Q(t)$ via the heat kernel
$$
 \rho_i(t) = e^{t\Delta}\star \rho_i(\cdot,0)+ \int_0^t e^{(t-s)\Delta} \star \left(-\sj a_{ij} \nabla_x \cdot (
 \rho_i(s)\nabla_x u_i(s))\right) ds
$$
 applied to $\rho_i^0$ and $\rho_i^1$
 to obtain 
$$
 Q(t) \leq C(P^0(t) + P^1(t)) Q(t).
$$
where $P^0, P^1$ stands for $\rho^0_i$ and $\rho_i^1$ respectively in 
(\ref{Pdef}).  This implies by \eqref{Pdef} that $Q(t)=0$ for 
$t>0$ for sufficiently small. 
Hence the uniqueness follows by iterating this method. 
\end{proof}

\section{Appendix}
\begin{lema}\label{ce}
The following inequalities hold true:
 \begin{itemize}
  \item[(a)] (Biler-Hebisch-Nadzieja type inequality \cite{Hybrid}) For any $\epsilon > 0$ there exists $L_{\epsilon} >0$
  such that 
  \begin{align*} 
   ||\rho||_{L^2(\rt)}^2 \leq \epsilon \left|\left| \frac{\nabla \rho}{\rho}\right|\right|_{L^2(
   \rt,\rho)}^2 ||
   \rho\ln\rho||_{L^1(\rt)} + L_{\epsilon}||\rho||_{L^1(\rt)}
  \end{align*}
  for all $\rho \in L^1_+(\rt)$ such that $\rho \ln \rho \in L^1(\rt)$ and $\frac{\nabla \rho}{\rho}
  \in L^2(\rt, \rho;\rt)$
   
\item[(b)] (Carleman Estimate \cite{BD}) For any $\rho \in L^1_+(\rt)$ if $\inte \rho \ln \rho \ dx< +\infty$ 
and $M_2(\rho) < +\infty$
then
\begin{align} \label{Carleman}
 \inte \rho|\ln \rho| \ dx \leq \inte \rho \ln \rho \ dx + M_2(\rho) + 2\ln(2\pi)\inte \rho \ dx + \frac{2}{e}.
\end{align}
\end{itemize}
\end{lema}

\subsection{Compactness Lemmas}
We will also use the following compactness result whose proof can be found in 
\cite[Theorem $5.4.4$]{AGS}:
\begin{propa}[Compactness of vector fields]\label{cvf}
 Let $\Omega$ be an open set in $\mathbb{R}^N$. If $\{\mu_m\}_m$ is a sequence of probability measures in $\Omega$
 narrowly converging to $\mu$ (in duality with $C_b(\mathbb{R}^N),$ continuous bounded functions) and $\{v_m\}_m$
 is a sequence of vector fields in $L^2(\Omega, \mu_m;\mathbb{R}^N)$ satisfying 
 \begin{align*}
  \sup_m ||v_m||_{L^2(\Omega,\mu_m;\mathbb{R}^N)} < +\infty,
 \end{align*}
then there exists a vector field $v \in L^2(\Omega, \mu;\mathbb{R}^N)$ such that 
\begin{align*}
 \lim_{m\rightarrow \infty} \int_{\Omega} \zeta \cdot v_m \ d\mu_m = \int_{\Omega} \zeta \cdot v \ d\mu, 
 \ \ for \ all \ \zeta 
 \in C_c^{\infty}(\Omega; \mathbb{R}^N)
\end{align*}
and satisfy 
\begin{align} \label{lscv}
 ||v||_{L^2(\Omega, \mu; \mathbb{R}^N)} \leq \liminf_{m\rightarrow \infty} ||v_m||_{L^2(\Omega, \mu_m; \mathbb{R}^N)}.
\end{align}
\end{propa}

We recall a refined Arzel\`{a}-Ascoli's compactness theorem obtained in \cite[Proposition $3.3.1$]{AGS}:
Let $(S,d)$ be a complete metric space and $\sigma$ be an Hausdorff topology on $S$ compatible with $d,$
in the sense that $\sigma$ is weaker than the topology induced by $d$ and $d$ is sequentially
$\sigma$-lower semicontinuous:
\begin{align*}
 x_i^m \xrightarrow[m \rightarrow \infty]{\sigma} x_i, \ i=1,2 \ \ \ \Rightarrow \ \ \ 
 d(x_1,x_2) \leq \ \liminf_{m \rightarrow \infty} \ d(x_1^m,x_2^m).
\end{align*}

\begin{thma}[Refined Arzel\`{a}-Ascoli]\label{ArzelaAscoli}
Let $T>0,$ let $K \subset \mathcal{S}$ be a sequentially compact set with respect to the topology $\sigma$ 
and let $u_m:[0,T]\rightarrow \mathcal{S}$ be curves such that 
\begin{align*}
 u_m(t) \in K, \ \ for \ all \ m \in \mathbb{N}, \ t \in [0,T], \\
 \limsup_{m\rightarrow \infty} d(u_m(s), u_m(t)) \leq w(s,t), \ \ for \ all \ s,t \in [0,T],
\end{align*}
for a symmetric function $w:[0,T]\times [0,T]\rightarrow [0,\infty),$ such that 
\begin{align*}
 \lim_{(s,t)\rightarrow (r,r)} w(s,t) = 0, \ \ for \ all \ r \in [0,T]\backslash \mathcal{B},
\end{align*}
where $\mathcal{B}$ is an (at most) countable subset of $[0,T].$ Then there exists an increasing
subsequence $p=m(p)$ and a limit curve $u:[0,T]\rightarrow \mathcal{S}$ such that 
\begin{align*}
 u_{m(p)}(t)\overset{\sigma}{\rightharpoonup} u(t), \ \ for \ all \ t \in [0,T], \ \ u \ is \ d-continuous \ in \ [0,T]\backslash \mathcal{B}.
\end{align*}
\end{thma}

The ensuing compactness lemma in $L^2(0,T;X)$ is a particular case of the compactness results obtained by 
J. Simon \cite[Lemma $9$]{Simon}:
\begin{lema}[Compactness in $L^2(0,T; X)$] \label{compactness in l2}
  Let  $X \subset B \subset Y$ be Banach spaces such that $X \subset B$ is compact. If a family 
 $F$ is bounded in $L^2(0,T;X)$ and relatively compact in $L^2(0,T;Y)$ then $F$ is relatively
 compact in  $L^2(0,T,B).$
\end{lema}

\bibliography{KS_system_Optimal_Transport} 

\newcommand{\etalchar}[1]{$^{#1}$}
\begin{thebibliography}{BKLN06b}

\bibitem[AF95]{AF}
Angela Alberico and Vincenzo Ferone.
\newblock Regularity properties of solutions of elliptic equations in {${\bf
  R}^2$} in limit cases.
\newblock {\em Atti Accad. Naz. Lincei Cl. Sci. Fis. Mat. Natur. Rend. Lincei
  (9) Mat. Appl.}, 6(4):237--250 (1996), 1995.

\bibitem[AGS05]{AGS}
Luigi Ambrosio, Nicola Gigli, and Giuseppe Savar\'e.
\newblock {\em Gradient flows in metric spaces and in the space of probability
  measures}.
\newblock Lectures in Mathematics ETH Z\"urich. Birkh\"auser Verlag, Basel,
  2005.

\bibitem[BA94]{Artzi}
Matania Ben-Artzi.
\newblock Global solutions of two-dimensional {N}avier-{S}tokes and {E}uler
  equations.
\newblock {\em Arch. Rational Mech. Anal.}, 128(4):329--358, 1994.

\bibitem[BCC08]{JKO2}
Adrien Blanchet, Vincent Calvez, and Jos\'{e}~A. Carrillo.
\newblock Convergence of the mass-transport steepest descent scheme for the
  subcritical {P}atlak-{K}eller-{S}egel model.
\newblock {\em SIAM J. Numer. Anal.}, 46(2):691--721, 2008.

\bibitem[BCC12]{BCCjfa}
Adrien Blanchet, Eric~A. Carlen, and Jos\'{e}~A. Carrillo.
\newblock Functional inequalities, thick tails and asymptotics for the critical
  mass {P}atlak-{K}eller-{S}egel model.
\newblock {\em J. Funct. Anal.}, 262(5):2142--2230, 2012.

\bibitem[BCK{\etalchar{+}}15]{Hybrid}
Adrien Blanchet, Jos\'{e}~Antonio Carrillo, David Kinderlehrer, Micha\l{}
  Kowalczyk, Philippe Lauren\c{c}ot, and Stefano Lisini.
\newblock A hybrid variational principle for the {K}eller-{S}egel system in
  {$\mathbb{R}^2$}.
\newblock {\em ESAIM Math. Model. Numer. Anal.}, 49(6):1553--1576, 2015.

\bibitem[BCM08]{BCM}
Adrien Blanchet, Jos\'{e}~A. Carrillo, and Nader Masmoudi.
\newblock Infinite time aggregation for the critical {P}atlak-{K}eller-{S}egel
  model in {$\mathbb{R}^2$}.
\newblock {\em Comm. Pure Appl. Math.}, 61(10):1449--1481, 2008.

\bibitem[BDEF10]{BDP}
Adrien Blanchet, Jean Dolbeault, Miguel Escobedo, and Javier Fern\'{a}ndez.
\newblock Asymptotic behaviour for small mass in the two-dimensional
  parabolic-elliptic {K}eller-{S}egel model.
\newblock {\em J. Math. Anal. Appl.}, 361(2):533--542, 2010.

\bibitem[BDP06]{BD}
Adrien Blanchet, Jean Dolbeault, and Beno\^{i}t Perthame.
\newblock Two-dimensional {K}eller-{S}egel model: optimal critical mass and
  qualitative properties of the solutions.
\newblock {\em Electron. J. Differential Equations}, pages No. 44, 32, 2006.

\bibitem[BG12]{two4}
Piotr Biler and Ignacio Guerra.
\newblock Blowup and self-similar solutions for two-component drift-diffusion
  systems.
\newblock {\em Nonlinear Anal.}, 75(13):5186--5193, 2012.

\bibitem[BKLN06a]{BT}
Piotr Biler, Grzegorz Karch, Philippe Lauren\c{c}ot, and Tadeusz Nadzieja.
\newblock The {$8\pi$}-problem for radially symmetric solutions of a chemotaxis
  model in a disc.
\newblock {\em Topol. Methods Nonlinear Anal.}, 27(1):133--147, 2006.

\bibitem[BKLN06b]{BKLN}
Piotr Biler, Grzegorz Karch, Philippe Lauren\c{c}ot, and Tadeusz Nadzieja.
\newblock The {$8\pi$}-problem for radially symmetric solutions of a chemotaxis
  model in the plane.
\newblock {\em Math. Methods Appl. Sci.}, 29(13):1563--1583, 2006.

\bibitem[BL13]{BL13}
Adrien Blanchet and Philippe Lauren\c{c}ot.
\newblock The parabolic-parabolic {K}eller-{S}egel system with critical
  diffusion as a gradient flow in {$\mathbb{R}^d,\ d\ge3$}.
\newblock {\em Comm. Partial Differential Equations}, 38(4):658--686, 2013.

\bibitem[Bla13]{Blanchet}
Adrien Blanchet.
\newblock On the parabolic-elliptic {P}atlak-{K}eller-{S}egel system in
  dimension 2 and higher.
\newblock In {\em S\'{e}minaire {L}aurent {S}chwartz---\'{E}quations aux
  d\'{e}riv\'{e}es partielles et applications. {A}nn\'{e}e 2011--2012},
  S\'{e}min. \'{E}qu. D\'{e}riv. Partielles, pages Exp. No. VIII, 26. \'{E}cole
  Polytech., Palaiseau, 2013.

\bibitem[Bre91]{Brenier}
Yann Brenier.
\newblock Polar factorization and monotone rearrangement of vector-valued
  functions.
\newblock {\em Comm. Pure Appl. Math.}, 44(4):375--417, 1991.

\bibitem[Bre94]{ABrezis}
Ha\"{i}m Brezis.
\newblock Remarks on the preceding paper by {M}. {B}en-{A}rtzi: ``{G}lobal
  solutions of two-dimensional {N}avier-{S}tokes and {E}uler equations''
  [{A}rch. {R}ational {M}ech. {A}nal. {\bf 128} (1994), no. 4, 329--358;
  {MR}1308857 (96h:35148)].
\newblock {\em Arch. Rational Mech. Anal.}, 128(4):359--360, 1994.

\bibitem[CD14]{CD}
Juan~F. Campos and Jean Dolbeault.
\newblock Asymptotic estimates for the parabolic-elliptic {K}eller-{S}egel
  model in the plane.
\newblock {\em Comm. Partial Differential Equations}, 39(5):806--841, 2014.

\bibitem[CEV11]{two3}
Carlos Conca, Elio Espejo, and Karina Vilches.
\newblock Remarks on the blowup and global existence for a two species
  chemotactic {K}eller-{S}egel system in {$\mathbb{R}^2$}.
\newblock {\em European J. Appl. Math.}, 22(6):553--580, 2011.

\bibitem[CL93]{CL93}
Wen~Xiong Chen and Congming Li.
\newblock Qualitative properties of solutions to some nonlinear elliptic
  equations in {${\bf R}^2$}.
\newblock {\em Duke Math. J.}, 71(2):427--439, 1993.

\bibitem[CP81]{Childress}
Stephen Childress and Jerome~K. Percus.
\newblock Nonlinear aspects of chemotaxis.
\newblock {\em Math. Biosci.}, 56(3-4):217--237, 1981.

\bibitem[CSW97]{CSW}
Michel~M. Chipot, Itai Shafrir, and Gershon Wolansky.
\newblock On the solutions of {L}iouville systems.
\newblock {\em J. Differential Equations}, 140(1):59--105, 1997.

\bibitem[EASV09]{two1}
Elio~Eduardo Espejo~Arenas, Angela Stevens, and Juan J.~L. Vel\'{a}zquez.
\newblock Simultaneous finite time blow-up in a two-species model for
  chemotaxis.
\newblock {\em Analysis (Munich)}, 29(3):317--338, 2009.

\bibitem[EASV10]{two2}
Elio~Eduardo Espejo~Arenas, Angela Stevens, and Juan J.~L. Vel\'{a}zquez.
\newblock A note on non-simultaneous blow-up for a drift-diffusion model.
\newblock {\em Differential Integral Equations}, 23(5-6):451--462, 2010.

\bibitem[EVC13]{two5}
Elio Espejo, Karina Vilches, and Carlos Conca.
\newblock Sharp condition for blow-up and global existence in a two species
  chemotactic {K}eller-{S}egel system in {$\mathbb{R}^2$}.
\newblock {\em European J. Appl. Math.}, 24(2):297--313, 2013.

\bibitem[FHM14]{NMS}
Nicolas Fournier, Maxime Hauray, and St\'{e}phane Mischler.
\newblock Propagation of chaos for the 2{D} viscous vortex model.
\newblock {\em J. Eur. Math. Soc. (JEMS)}, 16(7):1423--1466, 2014.

\bibitem[FM16]{FM}
Giani Ega\~{n}a Fern\'{a}ndez and St\'{e}phane Mischler.
\newblock Uniqueness and long time asymptotic for the {K}eller-{S}egel
  equation: the parabolic-elliptic case.
\newblock {\em Arch. Ration. Mech. Anal.}, 220(3):1159--1194, 2016.

\bibitem[Gro75]{Gross}
Leonard Gross.
\newblock Logarithmic {S}obolev inequalities.
\newblock {\em Amer. J. Math.}, 97(4):1061--1083, 1975.

\bibitem[Hor03]{Horsurvey1}
Dirk Horstmann.
\newblock From 1970 until present: the {K}eller-{S}egel model in chemotaxis and
  its consequences. {I}.
\newblock {\em Jahresber. Deutsch. Math.-Verein.}, 105(3):103--165, 2003.

\bibitem[Hor04]{Horsurvey2}
Dirk Horstmann.
\newblock From 1970 until present: the {K}eller-{S}egel model in chemotaxis and
  its consequences. {II}.
\newblock {\em Jahresber. Deutsch. Math.-Verein.}, 106(2):51--69, 2004.

\bibitem[Hor11]{H}
Dirk Horstmann.
\newblock Generalizing the {K}eller-{S}egel model: {L}yapunov functionals,
  steady state analysis, and blow-up results for multi-species chemotaxis
  models in the presence of attraction and repulsion between competitive
  interacting species.
\newblock {\em J. Nonlinear Sci.}, 21(2):231--270, 2011.

\bibitem[JKO98]{JKO}
Richard Jordan, David Kinderlehrer, and Felix Otto.
\newblock The variational formulation of the {F}okker-{P}lanck equation.
\newblock {\em SIAM J. Math. Anal.}, 29(1):1--17, 1998.

\bibitem[KS70]{KS}
Evelyn~F. Keller and Lee~A. Segel.
\newblock Initiation of slime mold aggregation viewed as an instability.
\newblock {\em J. Theoret. Biol}, 26:399--415, 1970.

\bibitem[KW]{KW}
Debabrata Karmakar and Gershon Wolansky.
\newblock On {L}iouville's systems corresponding to self similar solutions of
  the {K}eller-{S}egel systems of several populations.
\newblock {\em Preprint}.

\bibitem[McC97]{McCann}
Robert~J. McCann.
\newblock A convexity principle for interacting gases.
\newblock {\em Adv. Math.}, 128(1):153--179, 1997.

\bibitem[MMS09]{MMS}
Daniel Matthes, Robert~J. McCann, and Giuseppe Savar\'{e}.
\newblock A family of nonlinear fourth order equations of gradient flow type.
\newblock {\em Comm. Partial Differential Equations}, 34(10-12):1352--1397,
  2009.

\bibitem[NS98]{NSenba}
Toshitaka Nagai and Takasi Senba.
\newblock Global existence and blow-up of radial solutions to a
  parabolic-elliptic system of chemotaxis.
\newblock {\em Adv. Math. Sci. Appl.}, 8(1):145--156, 1998.

\bibitem[Ott98]{Otto}
Felix Otto.
\newblock Dynamics of labyrinthine pattern formation in magnetic fluids: a
  mean-field theory.
\newblock {\em Arch. Rational Mech. Anal.}, 141(1):63--103, 1998.

\bibitem[Pat53]{Pa}
Clifford~S. Patlak.
\newblock Random walk with persistence and external bias.
\newblock {\em Bull. Math. Biophys.}, 15:311--338, 1953.

\bibitem[San15]{San}
Filippo Santambrogio.
\newblock {\em Optimal transport for applied mathematicians}, volume~87 of {\em
  Progress in Nonlinear Differential Equations and their Applications}.
\newblock Birkh\"{a}user/Springer, Cham, 2015.
\newblock Calculus of variations, PDEs, and modeling.

\bibitem[Sim87]{Simon}
Jacques Simon.
\newblock Compact sets in the space {$L^p(0,T;B)$}.
\newblock {\em Ann. Mat. Pura Appl. (4)}, 146:65--96, 1987.

\bibitem[SS02]{SeSu}
Takasi Senba and Takashi Suzuki.
\newblock Weak solutions to a parabolic-elliptic system of chemotaxis.
\newblock {\em J. Funct. Anal.}, 191(1):17--51, 2002.

\bibitem[SS04]{SSbook}
Takasi Senba and Takashi Suzuki.
\newblock {\em Applied analysis}.
\newblock Imperial College Press, London, 2004.
\newblock Mathematical methods in natural science.

\bibitem[Sta63]{Stp}
Guido Stampacchia.
\newblock Some limit cases of {$L^{p}$}-estimates for solutions of second order
  elliptic equations.
\newblock {\em Comm. Pure Appl. Math.}, 16:505--510, 1963.

\bibitem[Suz05]{Subook}
Takashi Suzuki.
\newblock {\em Free energy and self-interacting particles}, volume~62 of {\em
  Progress in Nonlinear Differential Equations and their Applications}.
\newblock Birkh\"{a}user Boston, Inc., Boston, MA, 2005.

\bibitem[SW05]{SW}
Itai Shafrir and Gershon Wolansky.
\newblock Moser-{T}rudinger and logarithmic {HLS} inequalities for systems.
\newblock {\em J. Eur. Math. Soc. (JEMS)}, 7(4):413--448, 2005.

\bibitem[Vil03]{Villani03}
C\'{e}dric Villani.
\newblock {\em Topics in optimal transportation}, volume~58 of {\em Graduate
  Studies in Mathematics}.
\newblock American Mathematical Society, Providence, RI, 2003.

\bibitem[Wol02]{W1}
Gershon Wolansky.
\newblock Multi-components chemotactic system in the absence of conflicts.
\newblock {\em European J. Appl. Math.}, 13(6):641--661, 2002.

\end{thebibliography}

\bibliographystyle{alpha} 

\end{document}